\documentclass{scrartcl}

\usepackage{graphicx}     
\usepackage{amsfonts, amssymb, latexsym, mathtools}
\usepackage{amsmath}
\usepackage{amsthm}
\usepackage{graphicx}
\usepackage{epstopdf}
\usepackage{nicefrac}

\usepackage{color}

\newtheorem{defn}{Definition}

\newtheorem{proposition}[defn]{Proposition}
\newtheorem{definition}[defn]{Definition}

\newtheorem{remark}[defn]{Remark}
\newtheorem{lemma}[defn]{Lemma}
\newtheorem{corollary}[defn]{Corollary}
\newtheorem{theorem}[defn]{Theorem}
\newtheorem{assumption}[defn]{Assumption}

\DeclareMathOperator{\dist}{dist}
\usepackage[normalem]{ulem}
\usepackage{mathtools}
\usepackage{tikz}
\usetikzlibrary{positioning,matrix}
\usepackage{calc}
\newcommand{\blockbmatrix}[3][]{%
	\begin{tikzpicture}
		[
		baseline=-\the\dimexpr\fontdimen22\textfont2\relax,
		inner sep=1pt,
		outer sep=0pt,
		every left delimiter/.style={xshift=.2ex},
		every right delimiter/.style={xshift=-.2ex}
		]
		\matrix
		[
		matrix of math nodes,
		left delimiter={[},
		right delimiter={]},
		nodes={minimum width=1em},
		execute at begin cell={\mathstrut},
		execute at empty cell={\node{0};},
		ampersand replacement=\&
		] (m)
		{
			#2\\
		};
		
		\foreach \i/\j in {#1}
		{
			\draw[densely dotted,semithick] (m-\i-\i.north west) rectangle (m-\j-\j.south east);
		}
		
		#3
	\end{tikzpicture}%
}

\newcommand{\vertiii}[1]{{\vert\kern-0.25ex\vert\kern-0.25ex\vert #1 
		\vert\kern-0.25ex\vert\kern-0.25ex\vert}}

\newcommand{\mcl}[1]{\mathcal #1}

\begin{document}

\title{Manifold Turnpikes, Trims and Symmetries
}

\author{Timm Faulwasser\footnote{T.~Faulwasser,
		Institute for Energy Systems, Energy Efficiency and Energy Economics, TU Dortmund University,  Germany, e-mail: timm.faulwasser@ieee.org}
\and
Kathrin Fla\ss{}kamp\footnote{K. Fla{\ss}kamp, 
              Systems Modeling and Simulation, Saarland University, Germany, e-mail: kathrin.flasskamp@uni-saarland.de}
\and
Sina Ober-Bl\"{o}baum\footnote{S. Ober-Bl\"obaum,
	Department of Mathematics, University of Paderborn, Germany, e-mail: sinaober@math.upb.de}
\and
Manuel Schaller\footnote{M. Schaller,
	Institut f\"{u}r Mathematik,
	Technische Universit\"{a}t Ilmenau, Germany,
	email: manuel.schaller@tu-ilmenau.de. M.\ Schaller was funded by the DFG (grant WO\ 2056/2-1)}
\and
Karl Worthmann\footnote{K. Worthmann,
          Institut f\"{u}r Mathematik,
         Technische Universit\"{a}t Ilmenau, Germany,
        email: karl.worthmann@tu-ilmenau.de. K.\ Worthmann gratefully acknowledges funding by the German Research Foundation (DFG; grant WO\ 2056/6-1, project number 406141926)}}

\maketitle

\begin{abstract}                
Classical turnpikes correspond to optimal steady states which are attractors of optimal control problems. In this paper, motivated by mechanical systems with symmetries, we generalize this concept to manifold turnpikes. Specifically, the necessary optimality conditions on a symmetry-induced manifold coincide with those of a reduced-order problem under certain conditions. We also propose sufficient conditions for the existence of manifold turnpikes based on a tailored notion of dissipativity with respect to manifolds. We show how the classical Legendre transformation between Euler-Lagrange and Hamilton formalisms can be extended to the adjoint variables. Finally, we draw upon the Kepler problem to illustrate our findings.
\end{abstract}


\section{Introduction}
Studying the dynamics of classical mechanical systems has a long history.
In particular, the differential geometric viewpoint, which focuses on coordinate-invariant descriptions of mechanical system dynamics, is given, e.g., \ in \cite{MarsdenRatiu} and taken over to optimal control in \cite{BuLe04}.
Here, system dynamics are encoded either in the Lagrangian or the Hamiltonian function which lead to the well-known Euler-Lagrange equations or Hamilton equations, respectively.
Of particular interest in this context is the study of symmetries.
For mechanical systems, symmetries are characterized by an invariance of the Lagrangian with respect to translations or rotations of the system, for instance.
These symmetries can be described by actions of a Lie group.
Due to symmetry, equivalent trajectories exist, i.e., \ possibly controlled system trajectories which are identical modulo the action of the Lie group.
The close relation of symmetries and conserved quantities of dynamical systems goes back to Noether's fundamental insights obtained in the 1920s. 
Symmetry can be exploited to reduce the system dimension, see e.g.\ \cite{MarsdenRatiu} for an introduction to symmetry-reduction.
Dynamical systems with symmetry might show further structure in terms of relative equilibria, which are system motions that are completely generated from the symmetry action and thus, partially stationary in all other directions.
Relative equilibria are then obtained as  steady states of a symmetry reduced system.
This can clearly be seen in the historical setting of Routh, wherein symmetry is assumed as invariance w.r.t.\ a subset of all configuration states (see e.g.\ \cite{Bloch} for a concise introduction).
While, classically, relative equilibria have been studied for uncontrolled systems, Frazzoli et al.\ give a generalization termed \emph{trim primitive} \cite{FrDaFe05}.
Trim primitives can be exploited in the analysis of optimal control problems (OCP), in motion planning,  or in model predictive control of dynamical systems \cite{FrDaFe05,FOK12,Fla2013,FlasOber19}.

In the context of optimal control in economics, the notation of turnpike phenomena dates back to the foundational book of Dorfman~\cite{Dorfman58}, while earlier reference mentioning the phenomenon can be traced back to Ramsey~\cite{Ramsey28} or von Neumann~\cite{vonNeumann38}. The turnpike phenomenon refers to a similarity property of OCPs whereby for varying initial conditions and varying horizon lengths the optimal solutions approach the neighborhood of a specific steady state during the middle part of the horizon and the time spend close to this steady state (a.k.a. the turnpike) grows as the horizon increases.  Analysis and investigation of this concept  are a classical branch of optimal control for economics, cf.~\cite{Mckenzie76,Carlson91}. However, recently there has been a renewed interest in turnpike properties for optimal control of finite-and infinite-dimensional systems~\cite{Lance20a,Gugat19a,Gruene2019,Stieler14a} and in the context of receding-horizon solutions to OCPs~\cite{Gruene13a,kit:faulwasser18c}. Interestingly, there exists a close relation between turnpike properties and dissipativity notions in OCPs, see~\cite{Gruene16a,Faulwasser2017}. The main advantage of the dissipativity-based approach to turnpike results is that it allows to uncover fundamental mechanisms generating the turnpike phenomenon, see \cite{tudo:faulwasser20l} for a recent literature overview. This way, it goes beyond the economics inspired approach which identifies the phenomenon in specific problems and only rarely asked for generalized analysis. 

Moreover, it deserves to be noticed that the turnpike---i.e., the steady state which is approached be the optimal finite-horizon solutions and which under suitable conditions turns out to be a stable equilibrium of the infinite-horizon optimal solutions~\cite{tudo:faulwasser20a}---can be regarded as the attractor of the infinite-horizon OCP. Hence it is far from surprise that this attractor can be more general than a simple equilibrium. For example, in~\cite{Samuelson76} a periodic turnpike theorem is introduced. Moreover, in~\cite{FaulFlas19,tudo:faulwasser20e} we analysed a class of time-varying turnpike properties induced by symmetry in a specific class of OCPs formulated on Euler-Lagrange equations. Recently in \cite{Schaller2020a} it was shown that minimization of supplied energy in the context of port-Hamiltonian systems gives rise to an entire linear subspace of turnpikes.

The contribution of the present paper is to link the realms of turnpikes, trim solutions, and symmetries in OCPs for mechanical systems. 
Specifically, we consider Lagrangian systems with symmetries. Based on the established concepts of trim solutions, we show that if either one first formulates the OCP and then applies the trim condition to the optimality system, or one first applies the trim condition and then formulates a reduced OCP, one obtains the same result. While at first glance this looks not surprising, the commutativity of problem reduction and optimization generalizes a classical insight, wherein turnpikes are characterized as the attractive steady states of the optimality system~\cite{Trelat15a,kit:zanon18a,tudo:faulwasser20h,Gruene2019}. Specifically, this approach provides a handle to characterize time-varying turnpike solutions via a reduced OCP. Moreover, we show that under mild assumptions---i.e., if one allows non-equilibrium solutions travelling through the trim manifold---a dissipativity concept enables an elegant characterization. Specifically, we introduce a notion of dissipation of optimal solutions with respect to the distance to a manifold (here the trim manifold) and we show that this implies that optimal system operation indeed occurs on this manifold. Moreover, we show that the very same dissipativity condition implies the existence of a measure turnpike with respect to the trim manifold, i.e., the optimal solutions will spend only limited amount of time off the this manifold. 
In sum, the present paper does not only generalize our previous conference publications~\cite{FaulFlas19,tudo:faulwasser20e}, it also introduces a novel manifold generalization of the established dissipativity notion for OCPs.

The remainder of the paper is structured as follows: In Section \ref{sec:Preliminaries} we introduce the problem setting and provide background on symmetries and trims as well as on Lagrangian and Hamiltonian systems. Section \ref{sec:generalforcing} provides novel results on the equivalence of first applying a specific problem reduction and then optimizing with the reversed order sequence. Moreover, in this section we also introduce the concept of dissipativity of OCPs with respect to manifolds and we show that this allows to state sufficient conditions for a generalized turnpike property on manifolds, whereby the manifold may or may not be induced by an underlying symmetry. Section \ref{sec:trim_tp:H} links the former results to mechanical systems in Hamiltonian form and shows how one may elegantly map the adjoint variables between OCPs for Lagrangian and OCPs for Hamiltonian systems. Finally, Section \ref{sec:Kepler} illustrates our findings considering the Kepler problem. This paper ends with conclusions and an outlook.

\noindent{\textbf{Notation}: $\mathbb{N}$ denotes the positive integers, $\mathbb{N}_0 \doteq    \mathbb{N} \cup \{0\}$, and $\mathbb{R}$ represents the real numbers. $L^\infty([0,T],\mathbb{R}^n)$ is the space of Lebesgue-measurable, essentially bounded functions on the interval~$[0,T]$ mapping into~$\mathbb{R}^n$, $n \in \mathbb{N}$. Moreover, the Sobolev space $W^{1,\infty}([0,T],\mathbb{R}^n)$ is the linear space of all functions $x:[0,T] \rightarrow \mathbb{R}^n$ such that $x,\dot{x} \in L^\infty([0,T],\mathbb{R}^n)$ where $\dot{x}$ denotes the weak time derivative of~$x$. Furthermore, for $x \in \mathbb{R}^n$ and a nonempty set $S \subseteq \mathbb{R}^n$, $\dist(x,S)$ is defined by $\inf_{y \in S} \| x - y \|$ where $\| \cdot \|$ denotes the Euclidean distance in~$\mathbb{R}^n$.}

\section{Lagrangian systems with cyclic variables}\label{sec:Preliminaries}

%
%

Mechanical systems can be described by the Lagrange function
\begin{align}\label{eq:LagrangeFunction}
	L(q,\dot{q}) = \frac 12 \dot{q}^\top M(q) \dot{q} - V(q),
\end{align}
composed of the kinetic energy $\frac 12 \dot{q}^\top M(q) \dot{q}$ with symmetric mass matrix $M(q) \in \mathbb{R}^{n \times n}$ %
and potential energy~$V$. %
Here, the time-dependent configuration variables are denoted by $q = q(t) \in Q$, where $Q$ is the $n$-dimensional configuration manifold. The corresponding velocities $\dot{q} = \dot{q}(t)$ lie in the tangent space~$T_qQ$ at~$q$. The tangent bundle is denoted by~$TQ$, its dual, the cotangent bundle, by $T^*Q$.

The Euler-Lagrange equation
\begin{equation}\label{ELStandard}
	\frac {\mathrm{d}}{\mathrm{d}\,t} \frac {\partial}{\partial\,\dot{q}} L(q,\dot{q}) - \frac {\partial}{\partial\,q} L(q,\dot{q}) = f(u)
\end{equation}
with forcing term $f: \mathbb{R}^{m} \rightarrow {T^*Q}$, $m\in \mathbb{N}$, reads  
\begin{align}
	& & M(q) \ddot{q} + \nabla m(q,\dot{q}) \dot{q} - \frac 12 \dot{q}^\top \nabla m(q,\dot{q}) + \nabla V(q) & = f(u) \label{eq:ELreformulated}
\end{align}
where we used the abbreviation $\nabla m(q,\dot{q}) \doteq    \frac {\partial}{\partial\,q} \left(M(q) \dot{q}\right)$ to avoid tensor calculations. %

\subsection{Cyclic variables}\label{subsec:symmetry}


A subclass of (Lagrangian) systems exhibits \emph{cyclic variables}, which induce the following structure. The configuration space can be split into copies of $S^1$ and the so-called shape space $S$, i.e.,~$Q = S \times (S^1 \times \ldots \times S^1)$. Accordingly, the configuration variables $q$ are split into shape variables $s$ and cyclic variables~$\theta$, i.e., $q = (s, \theta)$, where the dimension of~$\theta$ defines the number of copies of~$S^1$. The variables $\theta$ are those variables, which do not appear explicitly in the kinetic and potential energy, although their velocities do. This directly implies that also the Lagrangian is independent of the \emph{cyclic variables} $\theta$,
\[ L(q,\dot{q}) = L(s,\dot{s},\dot{\theta})\]
Finally, the existence of cyclic variables induces symmetry in the unforced system, i.e.\ $f(u)\equiv 0$.
The Euler-Lagrange equations \eqref{ELStandard} for $\theta$ reduce to
\[
\frac {\mathrm{d}}{\mathrm{d}\,t}  \frac {\partial}{\partial\,\dot{\theta}} L(s,\dot{s},\dot{\theta}) = 0,
\]
i.e.\ $p_\theta \doteq \frac {\partial}{\partial\,\dot{\theta}} L(s,\dot{s},\dot{\theta}) $ is constant along system motions.
This is a special case of Noether's theorem, which relates system symmetry to the existence of a conserved quantity. In our case, the system symmetry is induced by the invariance of $L$ w.r.t.\ shifts in $\theta$, 

\subsection{Block-diagonal mass matrix structure}

As the mass matrix $M$ and the potential energy~$V$ are both independent of the cyclic variables~$\theta$ we may write
\begin{align}\nonumber
	M(q) = \begin{bmatrix}
		M_{11}(s) & M_{12}(s) \\
		M_{21}(s) & M_{22}(s)
	\end{bmatrix} \qquad\text{ and }\qquad V(q) = V(s).
\end{align}
In the subsequent analysis, we make the following assumption to avoid technicalities.
\begin{assumption}\label{Ass:BlockDiagonalForcing}
	Let the matrix~$M$ be block diagonal, i.e.\ $M_{12} = M_{21} = 0$ holds. Let further $M(q)=M(s)$ be regular for all $s\in S$.
\end{assumption}

Invoking Assumption~\ref{Ass:BlockDiagonalForcing} to make use of the block-diagonal inertia matrix~$M$, the Euler-Lagrange equation~\eqref{eq:ELreformulated} can be written as \begin{equation}
	\begin{bmatrix}
		M_{11} & 0 \\ 0 & M_{22}
	\end{bmatrix}\hspace*{-0.1cm} \begin{bmatrix}
		\ddot{s} \\ \ddot{\theta}
	\end{bmatrix} \hspace*{-0.05cm} + \hspace*{-0.05cm} \begin{bmatrix}
		\frac{\partial}{\partial s}(M_{11} \dot{s}) \dot{s} \\
		\frac{\partial}{\partial s} (M_{22} \dot{\theta}) \dot{s}
	\end{bmatrix} \hspace*{-0.05cm}-\hspace*{-0.05cm} \frac 12 \hspace*{-0.05cm}\begin{bmatrix}
		\dot{s}^\top \frac{\partial}{\partial s}(M_{11} \dot{s}) \hspace*{-0.05cm}+\hspace*{-0.05cm} \dot{\theta}^\top \frac{\partial}{\partial s} (M_{22} \dot{\theta}) \\ 0
	\end{bmatrix} \hspace*{-0.05cm}+\hspace*{-0.05cm} \begin{bmatrix}
		\frac {\partial V(s)}{\partial\,s} \\ 0
	\end{bmatrix} = \begin{bmatrix} f_s(u)\\ f_\theta(u)\end{bmatrix} , \nonumber
\end{equation}
where we suppressed the argument~$s$ of the matrices~$M_{ii}$, $i \in \{1,2\}$. Clearly, this can be written as a first-order system
\begin{equation}\label{eq:ELBlockDiag_vs}\tag{EL}
	\begin{split}
		\dot{s} & = v_s \\
		\dot{v}_s & = M_{11}^{-1}(s) \left(\textstyle \frac {v_s^\top}2 \frac{\partial (M_{11}(s) v_s)}{\partial s} + \frac {v_{\theta}^\top}2 \frac{\partial (M_{22}(s) v_\theta)}{\partial s} - \frac{\partial (M_{11}(s) v_s)}{\partial s} v_s - \frac {\partial V(s)}{\partial\,s} + f_s(u) \right)  \\
		\dot{\theta} & = v_\theta  \\
		\dot{v}_\theta & = M_{22}^{-1}(s) \left(-\frac{\partial}{\partial s} (M_{22}(s) v_\theta) v_s  + f_\theta(u) \right). 
	\end{split} 
\end{equation}

\begin{remark}[Non-orthogonal forcing]\label{rem:orthForcing}
	Note that what we observed in Sect.~\ref{subsec:symmetry}, i.e.\
	\begin{equation}\label{eq:conservedQuant}
		p_\theta \doteq \frac {\partial}{\partial\,v_\theta} L(s,\dot{s},v_\theta) = M_{22}(s) v_{\theta} =  const.,
	\end{equation}
	does not hold in \eqref{eq:ELBlockDiag_vs} in the presence of forces.
	However, in case of \emph{orthogonal forcing} (i.e.\ orthogonal to the subspace spanned by cyclic variables), namely $f_\theta(u)\equiv0$, $p_\theta$
	remains an invariant along solutions to \eqref{eq:ELBlockDiag_vs}.
\end{remark}

\subsection{Trim primitives}

Symmetry in Lagrangian systems may lead to the existence of special trajectories, so-called \emph{trim primitives} (\emph{trims} for short).

\begin{definition}[Trim Primitive] \label{def:Trims}
	Let a Lagrangian system \eqref{eq:ELBlockDiag_vs} be given.
	Assume orthogonal forcing,  i.e.\ $f_\theta \equiv 0$ holds.
	Then, a trajectory  of the \eqref{eq:ELBlockDiag_vs} system\\
	$(s,v_s,\theta,v_\theta)(t; (s^0,v_s^0,\theta^0,v_\theta^0))$
	emanating from the initial state $(s^0,v_s^0,\theta^0,v_\theta^0)$ with constant input $u(t) \equiv \bar{u}$, %
	is called a \textit{trim (primitive)} if it can be written for all $t\geq 0$ as 
	\begin{align} \label{eq:defTrim}
		\begin{split}
			s(t) & = s^0,\\
			v_s(t) & = 0,\\
			\theta(t) & = \theta^0 + v_\theta^0 \cdot t,\\
			v_\theta(t) & = v_\theta^0.
		\end{split}
	\end{align}
\end{definition}

Roughly speaking, a trim is a motion of constant velocity only in the direction of the cyclic variables while the shape space variables are constant.
Note that we can explicitly write down this specific type of trajectories, although we do not assume to know the solution of \eqref{eq:ELBlockDiag_vs} in general.

Formally, trim primitives are motions along the group orbits of the symmetry group
and they correspond to relative equilibria in the uncontrolled case; for details we refer to \cite{Fla2013,FOK12} and \cite[Section~2]{FlasOber19} for an illustrative example. %
We detail both characterizations in the following.

\subsubsection{Trim characterization via controlled potentials}
Relative equilibria can be found by solving for the critical points of the \emph{amended potential} \cite{MarsdenRatiu}.
In \cite{Fla2013,FOK12}, this approach has been extended to controlled potentials in order to compute trim primitives.
Consider the function $\nu(s,u) \doteq    s^\top  f_s(u)$. Clearly,
$ 
\frac{\partial \nu}{\partial s} = f_s(u)
$
holds. Then, we define the \emph{forced potential}
\begin{equation}\label{eq:forcedPotential}
	V^u(s,u) \doteq    V(s) - \nu(s,u).
\end{equation}
Note that the Euler-Lagrange equations~\eqref{eq:ELBlockDiag_vs} with orthogonal forcing can also be derived as unforced Euler-Lagrange equations with $V^u$ replacing the original potential $V$.
Interpreting the forcing as an additional parametrized potential allows us to apply the classical theory of relative equilibria \cite{MarsdenRatiu}: %
The locked inertia tensor, i.e.\ the inertia tensor
if shape variables~$s$ are fixed,  is given by~$M_{22}(s)$. %
Then, we fix a value $\mu \,  \in T^\star_q Q$ of the conserved quantity \eqref{eq:conservedQuant}, $p_\theta = M_{22}(s) v_{\theta}$.
For instance, using the initial values~$s^0$ and~$v_\theta^0$ yields $\mu = M_{22}(s^0) v_{\theta}^0$. %
Lastly, we define the \emph{forced amended potential} as
\begin{align}\label{Eq:ForcedAmendedPotential}
	V_\mu^u(s,\mu,u) & \doteq    V^u(s,u) + \frac 12 \mu^\top M_{22}^{-1}(s) \mu. 
\end{align}
The following central lemma yields a way to derive trim primitives. A proof can be found in \cite{Fla2013,FOK12}.
\begin{lemma}\label{LemmaTrimPrimitives}
	Consider a Lagrangian system \eqref{eq:ELBlockDiag_vs} with orthogonal forcing, i.e.\ $f_\theta \equiv 0$.
	Let $(\hat{s},\hat{\mu},\hat{u})\in S\times T^\star_q Q\times \mathbb{R}^{m}$ 	
	be a critical point of the forced amended potential \eqref{Eq:ForcedAmendedPotential}, i.e.
	\begin{align}\label{eq:trimCritPointVumu}
		\nabla_s V_\mu^u(\hat{s},\hat{\mu},\hat{u}) = \frac{\partial V(s)}{\partial s} + \frac 12 \mu^\top \frac {\partial}{\partial s} (M^{-1}_{22}(s) \mu) -f_s(u) = 0.
	\end{align}
	Then, 	$(\hat{s},\hat{\mu},\hat{u})$ defines a trim primitive in the following way:
	\begin{align*}
		s(t) & = \hat{s}, &
		v_s(t)& = 0,\\
		\theta(t) & = \theta^0 + M_{22}(\hat{s})^{-1} \hat{\mu} \cdot t, &
		v_\theta(t) & =  M_{22}(\hat{s})^{-1} \hat{\mu},\\
		u(t) & = \hat{u},
	\end{align*}
	with $\theta^0$ being an arbitrary initial value of the cyclic variable.
\end{lemma}

\subsubsection{Trim characterization via partial steady states}
A trim requires the shape space variables $s$ to be at steady state.
Then, $\dot{s} = v_s = 0$ holds along the trim trajectory and thus, also $\dot{v}_s = 0$. This simplifies the corresponding differential equation in \eqref{eq:ELBlockDiag_vs} to (recall $M_{11}$ is assumed to be regular for all $s$)
\begin{equation*}
	\frac 12 v_\theta^\top \frac{\partial }{\partial s} (M_{22}(s) v_\theta) - \frac{\partial}{\partial s} V(s) + f_s(u) = 0.
\end{equation*}


\begin{lemma}[Trim condition]\label{lemma:T}
	Consider a Lagrangian system \eqref{eq:ELBlockDiag_vs} with orthogonal forcing, i.e.\ $f_\theta \equiv 0$.
	Let a function $T$ be given by
	\begin{align} \label{eq:T}
		T(s,v_\theta, u)\doteq  M_{11}^{-1}(s) \left(   \frac 12 v_\theta^\top \frac{\partial }{\partial s} (M_{22}(s) v_\theta) - \frac{\partial}{\partial s} V(s) + f_s(u) \right).
	\end{align}
	Then, a triple $(\tilde{s},\tilde{v}_\theta,\tilde{u})$ which satisfies $T(\tilde{s},\tilde{v}_\theta,\tilde{u}) = 0$ defines a trim primitive when setting $s^0=\tilde{s}$, $v_\theta^0=\tilde{v}_\theta$ and $\bar{u}=\tilde{u}$ in Definition~\ref{def:Trims},
	with $\theta^0$ being an arbitrary initial value of the cyclic variable.
\end{lemma}

\begin{corollary}\label{cor:VmuT}
	Every triple $(\hat{s},\hat{\mu},\hat{u})$ which satisfies \eqref{eq:trimCritPointVumu} in Lemma~\ref{LemmaTrimPrimitives}, also satisfies\\ $T(\hat{s},\hat{v}_\theta,\hat{u})=0$ with $\hat{v}_\theta = M_{22}(\hat{s})^{-1} \hat{\mu}$ in Lemma~\ref{lemma:T} and vice versa, using
	\begin{align} \label{eq:MassInv}
		M^\prime_{22}(s)^{-1} = - M_{22}^{-1}(s) M^\prime_{22}(s) M_{22}^{-1}(s).
	\end{align}
\end{corollary}

Finally, we introduce the \emph{trim manifold} as a $4n$-dimensional manifold in $TQ$.

\begin{definition}[Trim manifold]\label{rem:TrimManifold}
	The trim manifold $\mcl{T}$ is defined by
	\[
	\mcl{T} \doteq    \{ (s, v_s, \theta, v_\theta)^\top \in TQ \,|\, v_s = 0 \text{ and }  \exists u \in \mathbb{R}^{m}:~T(s,v_\theta,u) = 0 \},
	\]
	with $T$ defined by~\eqref{eq:T}, i.e., the manifold of states for which a control exists such that Definition~\ref{def:Trims} for trim primitives is satisfied. %
\end{definition}

Assuming orthogonal forcing ($f_\theta \equiv 0$) in
\eqref{eq:ELBlockDiag_vs}, each point in the trim manifold can serve as the initial value of a trim primitive.
If an initial point $(s^0, v_s^0, \theta^0, v_\theta^0) \in \mcl{T}$ and $\bar{u}$ are chosen such that $T(s^0,v_\theta^0,\bar{u})=0$ holds, then $(s,v_s,\theta,v_\theta)(t; (s^0,v_s^0,\theta^0,v_\theta^0)) \in \mcl{T}$ for all $t\geq 0$, i.e.\ the trim primitive stays within the trim manifold.

\begin{remark}\label{rem:outputZeroDynamics}
	Moreover, the function $T$ can be interpreted as an output map and the dynamics on $\mcl{T}$ can be understood as the \emph{zero dynamics} of \eqref{eq:ELBlockDiag_vs} w.r.t.\ $T(s,v_\theta, u) = 0$.
	Note that this viewpoint does not rely on the restriction to orthogonal forcing, i.e.\ allowing $f_\theta (u) \neq 0$, more solutions besides trims which evolve in $\mcl{T}$ may exist. For more details on zero-dynamics, see~\cite{Nijmeijer90a,Isidori95a} and \cite{Olfati02a} for the link to symmetries.
\end{remark}

In the following, we study the role of trim primitives and output-zeroing solutions in optimal control.
To simplify the further exposition, we will consider all variables to be scalar-valued, i.e.~$(s,\theta,v_s,v_\theta) \in \mathbb{R}^4$ and $M_{11}(s), M_{22}(s) \in \mathbb{R}$. 
Note that by fixing the set of coordinates, we also decide to leave the differential geometric setting and consider coordinates in the state space $\mathbb{R}^4$.
This leads to the following set of system equations
%
\begin{align}
	\begin{split}
		\dot{s} & = v_s, \\
		\dot{v}_s& = M_{11}^{-1} (s)\left( \frac 12 M'_{22}(s)  v_\theta^2(t) - \frac 12 M'_{11}(s) v^2_s-  V'(s) + f_s(u) \right) \label{eq:v_s} \\
		\dot{\theta} & = v_\theta  \\
		\dot{v}_\theta & = M_{22}^{-1}(s) \left(- M'_{22}(s) v_\theta(t) v_s + f_\theta(u) \right).
	\end{split}
\end{align}

Note that in terms of the trim conditions, we may write $ \dot{v}_s = T(s,v_\theta,u)-  \frac 12 M_{11}^{-1} M'_{11}(s) v^2_s $, where the last term vanishes on $\mathcal{T}$.

%

\section{Optimal control of systems with symmetry}\label{sec:generalforcing}

We start by formulating the Optimal Control Problem (OCP) subject to the first-order Euler-Lagrange system \eqref{eq:v_s}. To this end, consider a continuously differentiable stage cost $\ell:\mathbb{R}^3 \times \mathbb{R}^m \to \mathbb{R}$. Further, let an initial state $(s^0,v_s^0,\theta^0,v_\theta^0)\in \mathbb{R}^4$ and a time horizon $T>0$ be given. The OCP we will consider in the following is given by \\
\noindent \fbox{
	\begin{minipage}{.98 \textwidth}
		\begin{align}\label{eq:OCP_full}\tag{OCP}
			\begin{split}
				\min_{u \in L^\infty([0,T],\mathbb{R}^{m})}\quad &\int_0^T \ell (s(t),v_s(t),v_\theta (t),u(t))\,\mathrm{d}t\\
				\text{subject to the system dynamics}&~\eqref{eq:v_s}, 
				\nonumber \\
				(s\ v_s\ \theta\ v_\theta)(0) &= (s^0\ v_s^0\ \theta^0\ v_\theta^0).
			\end{split}
		\end{align}
\end{minipage}}

Our standing assumption on the stage cost is as follows.
\begin{assumption}\label{Ass:ell}
	The stage cost is independent of~$\theta$, i.e., $\frac{\partial \ell}{\partial \theta} = 0$ holds. 
\end{assumption}
This assumption reflects the underlying symmetry property, i.e., we do not penalize $\theta$ in the OCP.

We will now derive the first-order Necessary Conditions of Optimality (NCO) for \eqref{eq:OCP_full}. As the OCP does not involve terminal constraints, singular arcs can not occur, cf.~\cite[Rem.\ 6.9, p.\ 168]{Locatelli01}, which allows us to normalize (and hence omit) the multiplier of the stage cost in the following.
We consider the adjoints (co-states)  $\lambda=\left(\lambda_s,\lambda_{v_s},\lambda_\theta,\lambda_{v_\theta}\right)$, 
and the (optimal control) Hamiltonian of \eqref{eq:OCP_full},
\begin{align*}
	\mathcal{H}(&s,v_s,v_\theta,u,\lambda) \doteq    \ell(s,v_s,v_\theta,u) \\ 
	& + \begin{bmatrix}
		\lambda_s \\ \lambda_{v_s} \\ \lambda_\theta \\ \lambda_{v_\theta}
	\end{bmatrix}^\top \begin{bmatrix}
		v_s\\
		M^{-1}_{11}(s) \left( \frac 12 M'_{22}(s)  v_\theta^2 - \frac 12M'_{11}(s) v_s^2 -V'(s) + f_s(u) \right) \\
		v_\theta  \\
		M^{-1}_{22}(s) \left(- M'_{22}(s) v_\theta v_s + f_\theta(u) \right)
	\end{bmatrix}.
\end{align*}
Hence the adjoint equations read
\begin{subequations}\label{eq:OCP_full_NCO}
	\begin{align}\label{eq:adjoint_fullOCP}
		\dot{\lambda}(t) = & -\begin{bmatrix}
			0 & 1& 0& 0\\
			L_1(s) & -M^{-1}_{11}(s) M'_{11}(s) v_s & 0 & M^{-1} _{11}(s) M'_{22}(s) v_\theta & \\
			0 &  0 & 0 & 1\\
			L_2(s)& -M^{-1}_{22}(s) M'_{22}(s) v_\theta & 0 & - M^{-1}_{22}(s) M'_{22}(s) v_s 
		\end{bmatrix}^\top \begin{bmatrix}
			\lambda_s \\ \lambda_{v_s} \\ \lambda_\theta \\ \lambda_{v_\theta}
		\end{bmatrix} - \begin{bmatrix}
			\frac{\partial \ell}{\partial s} \\
			\frac{\partial \ell}{\partial v_s}\\
			0 \\
			\frac{\partial \ell}{\partial v_\theta}
		\end{bmatrix}
	\end{align}
	with $\lambda(T) = 0$, where $L_1(s)$ and $L_2(s)$ are defined by
	\begin{align*}
		L_1(s) &\doteq     M^{-1}_{11}(s)' \left( - \frac 12 M'_{11}(s) v_s^2 + \frac 12  M'_{22}(s) v_\theta^2 - V'(s) + f_{s}(u)\right) \\
		&\quad+ M^{-1}_{11}(s)\left( - \frac 12M''_{11}(s) v_s^2 + \frac 12 M''_{22}(s)  v_\theta^2 -V''(s) \right) ,\\
		& = \frac{\partial}{\partial s} \left( T(s,v_\theta,u)-  \frac 12 M_{11}^{-1} M'_{11}(s) v^2_s \right),\\
		L_2(s)&\doteq   -M^{-1}_{22}(s)' M'_{22}(s) v_\theta v_s - M^{-1}_{22}(s)  M''_{22}(s) v_\theta v_s  +  M_{22}^{-1}(s)' f_\theta(u).
	\end{align*}
	Observe that 
	\[
	M^{-1} _{11}(s) M'_{22}(s) v_\theta =  \frac{\partial}{\partial v_\theta}  T(s,v_\theta,u).
	\]
	Moreover, the gradient stationarity condition reads
	\begin{align}\label{eq:gradeq}
		0 =  \mathcal{H}_u =& 
		M^{-1}_{11}(s) f_s^\prime(u) \lambda_{v_s} + M^{-1}_{22}(s) f_\theta^\prime(u) \lambda_{v_\theta} + \frac{\partial \ell}{\partial u}\\
		=&    \frac{\partial}{\partial u}  T(s,v_\theta,u)\lambda_{v_s} + M^{-1}_{22}(s) f_\theta^\prime(u) \lambda_{v_\theta} + \frac{\partial \ell}{\partial u}. \nonumber
	\end{align}
\end{subequations}

In order to identify turnpike phenomena in optimal control of Lagrangian systems of type \eqref{eq:v_s}, we will study the optimal control problem when restricting it to the trim manifold $\mcl{T}$ via its necessary conditions of optimality.

\subsection{Optimal control on trim manifold $\mcl{T}$}\label{Subsection:OCtrimManifold}

We are interested in studying the output zeroing dynamics as introduced in Remark~\ref{rem:outputZeroDynamics} in an optimal control setting.
In the scalar case, the output zeroing condition \eqref{eq:T} can be rewritten as
\begin{align}\label{eq:OutputZeroingCondition}
	T(s,v_\theta,u) = M_{11}^{-1}(s) \left( \frac 12 M^\prime_{22}(s)v^2_\theta -  V'(s) + f_s(u)\right) =0
\end{align}
If \eqref{eq:OutputZeroingCondition} holds along a trajectory with shape variable~$s$ being at steady state, i.e., $v_s(t) = 0$ holds for all $t \geq 0$, the solution stays within the trim manifold (cf.\ Definition~\ref{rem:TrimManifold}).
The trajectory is not necessarily a trim, though, since the control need not be constant.

Restricting to output-zeroing dynamics leads to the following reduced optimal control problem on the trim manifold~$\mcl{T}$

\noindent \fbox{
	\begin{minipage}{.98 \textwidth}
		\begin{align}\label{eq:T-OptimizationProblem}\tag{$\mathcal{T}$-OCP}
			\begin{split} 
				\min_{\bar{u} \in L^\infty([0,T],\mathbb{R}^{m}),\,\bar{s} \in \mathbb{R}}\quad & \int_0^T \ell(\bar{s},0,\bar{v}_{\bar{\theta}},\bar{u})\,\mathrm{d}t \\
				\text{subject to } \dot{\bar{\theta}}&= \bar{v}_{\bar{\theta}},\\
				\dot{\bar{v}}_{\bar{\theta}} & = M_{22}^{-1}(\bar{s}) f_\theta(\bar{u}), \\
				(\bar{{{\theta}}}\ \bar{v}_{\bar{\theta}})(0) &=(\theta^0\ v_\theta^0),\\ 
				T(\bar{s},\bar{v}_{\bar{\theta}},\bar{u}) & = 0 \quad  \forall t \in [0,T], \text{ as in } \eqref{eq:OutputZeroingCondition}.
			\end{split}
		\end{align}
\end{minipage}}

$\quad$\\
Next, we derive the NCO of \eqref{eq:T-OptimizationProblem}. %
To this end, we define the (optimal control) Hamiltonian 
\begin{equation}\nonumber
	\mathcal{H}(\bar{v}_{\bar{\theta}},\bar{u},\bar{s},\bar{\lambda}_{\bar \theta},\bar{\lambda}_{\bar{v}_{\bar{\theta}}})\doteq    \ell(\bar{s},0,\bar{v}_{\bar{\theta}},\bar{u}) + \bar{\lambda}_{\bar{\theta}} \bar{v}_{\bar{\theta}} + \bar{\lambda}_{\bar{v}_{\bar{\theta}}} M_{22}^{-1}(\bar{s}) f_\theta(\bar{u})
\end{equation}
with the adjoints $\bar{\lambda} = (\bar{\lambda}_{\bar{\theta}}, \bar{\lambda}_{\bar{v}_{\bar{\theta}}})$. %
Since the output zeroing condition~\eqref{eq:OutputZeroingCondition} is not included in this Hamiltonian, we augment it by direct adjoining. %
This is also known as the Lagrange formalism, see, e.g.~\cite{Hartl95}. The resulting (optimal control) Lagrangian is given by
\begin{equation}\nonumber
	\mathcal{L}(\bar{v}_{\bar{\theta}},\bar{u},\bar{s},\bar{\lambda}_{\bar{\theta}},\bar{\lambda}_{\bar{v}_{\bar{\theta}}},\bar{\lambda}_{\mcl{T}}) \doteq \mathcal{H}(\bar{v}_{\bar{\theta}},\bar{u},\bar{s},\bar{\lambda}_{\bar{\theta}},\bar{\lambda}_{\bar{v}_{\bar{\theta}}}) + \bar{\lambda}_{\mcl{T}} T(\bar{s},\bar{v}_{\bar{\theta}},\bar{u})
\end{equation}
with the Lagrange multiplier $\bar{\lambda}_{\mcl{T}}$ associated with the output zeroing condition~\eqref{eq:OutputZeroingCondition}. %
Since the right-hand sides of the Hamiltonian~$\mathcal{H}$ and the Lagrangian are independent of~$\bar{\theta}$, we suppress~$\bar{\theta}$ in the lists of arguments. %
%
Stationarity of the Lagrangian, i.e.\ $\nabla_{(\bar{u},\bar{s})} \mcl{L} = 0$, yields the NCO of ~\eqref{eq:T-OptimizationProblem}:
\begin{subequations}\label{eq:steadystate}
	\begin{align}
		\dot{\bar{\lambda}}_{\bar{\theta}}  & = \phantom{-}0 \label{eq:steadystate1} \\
		\dot{\bar{\lambda}}_{\bar{v}_{\bar{\theta}}} & = - \frac{\partial \ell}{\partial \bar{v}_{\bar{\theta}}}
		\hspace*{3.25cm}- \frac{\partial}{\partial \bar{v}_{\bar{\theta}}} T(\bar{s},\bar{v}_{\bar{\theta}},\bar{u}) \bar{\lambda}_{\mcl{T}}  \label{eq:steadystate2} \\
		0 & = \phantom{-}\frac {\partial \ell}{\partial \bar{u}} + \frac{\partial}{\partial \bar{u}} f_{{\theta}}(\bar{u}) M_{22}^{-1}(\bar{s}) \bar{\lambda}_{\bar{v}_{\bar{\theta}}} 
		+  \frac{\partial}{\partial \bar{u}} T(\bar{s},\bar{v}_{\bar{\theta}},\bar{u}) \bar{\lambda}_{\mcl{T}} \label{eq:steadystate3} \\		
		\begin{split}  \label{eq:steadystate4}
			0 & =\phantom{-} \frac {\partial \ell}{\partial \bar{s}} + \bar{\lambda}_{\bar{v}_{\bar{\theta}}} M^{-1}_{22}(\bar{s})' f_{{\theta}}(\bar{u}) 
			\hspace*{0.4cm}+  \frac{\partial}{\partial \bar{s}} T(\bar{s},\bar{v}_{\bar{\theta}},\bar{u})  \bar{\lambda}_{\mcl{T}}
		\end{split} \\
		T(\bar{s},\bar{v}_{\bar{\theta}},\bar{u}) &= \phantom{-}M_{11}^{-1}(s) \left( \frac12M_{22}'(\bar{s})\bar{v}_{\bar{\theta}}^2 - V'(\bar{s}) + f_s(\bar{u})\right) = 0, \label{eq:steadystate5} \\
		\dot{\bar{\theta}} &= \phantom{-}\bar{v}_{\bar{\theta}} \label{eq:steadystate6} \\
		\dot{\bar{v}}_{\bar{\theta}} &= \phantom{-}M_{22}(\bar{s})^{-1} f_\theta(\bar{u}) \label{eq:steadystate7}
	\end{align}
\end{subequations}
for a.e.\ $t\in [0,T]$, $\bar{\lambda}(T)=0$ for all adjoint states and $(\bar{\theta},\bar{v}_{\bar{\theta}})(0)=(\theta^0,v_\theta^0)$.
Then, in view of the terminal condition $\bar{\lambda}_{\bar{\theta}}(T)=0$ and the associated adjoint equation \eqref{eq:steadystate1}, %
we directly obtain $\bar{\lambda}_{\bar{\theta}} \equiv 0$ on~$[0,T]$. %

\subsection{Relation of \eqref{eq:OCP_full} and \eqref{eq:T-OptimizationProblem} via their NCO} 

The main result in this section is the following connection between the NCO of the full and the reduced optimal control problems \eqref{eq:OCP_full} and \eqref{eq:T-OptimizationProblem}.
\begin{proposition}[Correspondence of NCOs]\label{prop:NCOmatch}
	Consider \eqref{eq:OCP_full} and its reduced counterpart \eqref{eq:T-OptimizationProblem} for a Lagrangian system of type \eqref{eq:v_s}. %
	Suppose that Assumptions~\ref{Ass:BlockDiagonalForcing} and~\ref{Ass:ell} hold. %
	If an optimal solution and the corresponding Lagrange multiplier for \eqref{eq:OCP_full} satisfy $\dot{s}^\star(t) = v_s^\star(t) = 0$ and $\dot{\lambda}^\star_s(t) = 0$  for $t \in [t_1,t_2]\subseteq [0,T]$,
	then 
	they also satisfy the dynamics of the first-order NCOs of ~\eqref{eq:T-OptimizationProblem} for all $t\in [t_1,t_2]$ in the sense of Table~\ref{tab:ident1}. 
	
	Conversely, an optimal solution and the corresponding Lagrange multiplier of~\eqref{eq:T-OptimizationProblem} satisfy the dynamics of the first-order NCOs of \eqref{eq:OCP_full} in the sense of Table~\ref{tab:ident1} and setting $v_s(t) = 0$ and ${\lambda}_s(t) = 0$.
\end{proposition}
\begin{proof}
	We start with analyzing the NCO \eqref{eq:OCP_full_NCO}.
	With $\dot{s}^\star(t) = v_s^\star(t) = 0$ for an optimal solution of~\eqref{eq:OCP_full}, the primal dynamics \eqref{eq:v_s} are equivalent to \eqref{eq:steadystate5}-\eqref{eq:steadystate7}, i.e.,\ the primal dynamics of the NCO~\eqref{eq:T-OptimizationProblem} plus the output-zeroing condition $T(s,v_\theta,u) =0$.

	\begin{table}[!t]
		\centering
		\caption{Identification of the variables occuring in the NCOs of the reduced problem \eqref{eq:T-OptimizationProblem} and the full problem  \eqref{eq:OCP_full}. \label{tab:correspondence}}
		\begin{tabular}{ l l l }
			NCOs of \eqref{eq:T-OptimizationProblem}& NCO of \eqref{eq:OCP_full}  & Identification\\\hline \\[-.2cm]
			$-$ & $\lambda_s$& $0 = \dot{\lambda}_s$ \\ 
			$\bar{\lambda}_{\mcl{T}}$ & $\lambda_{v_s}$& $\bar{\lambda}_{\mcl{T}} = \lambda_{v_s}$\\  
			$\bar{\lambda}_{\bar\theta} $ & $\lambda_\theta$& $ \bar{\lambda}_{\bar\theta}  =\lambda_\theta =  0$ (since $\frac{\partial\ell}{\partial \theta} =0$)  \\
			$\bar{\lambda}_{\bar v_{\bar\theta}}$ & $\lambda_{v_\theta}$& $\bar{\lambda}_{\bar v_{\bar\theta}} =\lambda_{v_\theta}$\\
			$\bar{s}$ & $s$& $\dot{\bar{s}}=\dot{s}=0$ (since $s = const.$)\\
			$-$ & $v_s$ & $0 = v_s$\\
			$\bar{\theta}$ & $\theta$ & $\bar{\theta}=\theta$\\
			$\bar{v}_{\bar{\theta}}$ & $v_\theta$ & $\bar{v}_{\bar{\theta}} =  v_\theta$  (since $v_s=0$)\\
			$\bar{u}$&$u$&$\bar{u}=u$
		\end{tabular}
		\label{tab:ident1}
	\end{table}

	Observe that in the adjoint equations \eqref{eq:adjoint_fullOCP}, we have $\dot\lambda_\theta = 0$.
	Thus, with $\lambda_\theta(T) = 0$, we have $\lambda_\theta \equiv 0$.
	Applying the trim manifold constraints $s \equiv const.$ and  $v_s \equiv 0$ to the NCO~\eqref{eq:OCP_full_NCO} yields, 	
	suppressing the time argument for all functions,
	\begin{subequations}
		\label{eq:dynamicsys}
		\begin{align}
			\label{eq:dynamicsys1}
			\dot{\lambda}_s&= - \frac{\partial}{\partial s}T(s,v_\theta,u) \lambda_{v_s} - M_{22}^{-1}(s)'f_\theta(u)  \lambda_{v_\theta}-	\frac{\partial \ell}{\partial s} \\
			\label{eq:dynamicsys2}
			\dot{\lambda}_{v_s}&=	-\lambda_s + M_{22}^{-1}(s) M'_{22}(s)  v_\theta \lambda_{v_\theta} -\frac{\partial \ell}{\partial v_s}\\
			\label{eq:dynamicsys4}
			\dot{\lambda}_{v_\theta} &= - M_{11}(s)^{-1}  M'_{22}(s)v_\theta \lambda_{v_s}- \frac{\partial \ell}{\partial v_\theta}\\
			\label{eq:dynamicsys5}
			0&=
			\phantom{-}\frac{\partial}{\partial u}  T(s,v_\theta,u)\lambda_{v_s} + M^{-1}_{22}(s) f_\theta^\prime(u) \lambda_{v_\theta} + \frac{\partial \ell}{\partial u} 
		\end{align}
	\end{subequations}
	with
	\begin{align*}
		\frac{\partial}{\partial s}T(s,v_\theta,u)  = & M_{11}^{-1}(s)'\left(\frac12M'_{22}(s) v_\theta^2 - V'(s)+f_s(u)\right) \\
		& +M^{-1}_{11}(s) \left(\frac12M''_{22}(s)v_\theta^2-V''(s)\right).
	\end{align*}
	
	If $\bar{\lambda}_{\bar{v}_{\bar{\theta}}} = \lambda_{v_\theta}$ and
	$\bar{\lambda}_{\mathcal{T}}= \lambda_{v_s}$, then the gradient stationary condition \eqref{eq:dynamicsys5} is equivalent to \eqref{eq:steadystate3}.
	Then, 	with $\bar{\lambda}_{\mathcal{T}}=\lambda_{v_s}$ we see that \eqref{eq:dynamicsys4} coincides with \eqref{eq:steadystate2}.
	Further, if $\dot{\lambda}_s = 0$, \eqref{eq:dynamicsys1} is equivalent to \eqref{eq:steadystate4}, where we use that the primal variables of the full OCP are constrained on $\mathcal{T}$, i.e.,\ $T(s,v_\theta,u) =0$.
	Finally, \eqref{eq:dynamicsys2} does not have a counterpart in the NCOs \eqref{eq:steadystate} of the reduced system, since $\bar{\lambda}_{\mathcal{T}}$ is not an adjoint variable but a Lagrangian multiplier only.
	Nevertheless, the identifications show that given a solution of the NCO of \eqref{eq:T-OptimizationProblem}, this solution also satisfies the NCO of \eqref{eq:OCP_full} with $\dot{s}^\star(t) = v_s^\star(t) = \dot{\lambda}^\star_s(t) = 0$ and vice versa.
	Table~\ref{tab:correspondence} summarizes the correspondence of NCO variables for \eqref{eq:OCP_full} and \eqref{eq:T-OptimizationProblem}.
\end{proof}


The result shows that one may commutate applying the output-zeroing condition~\eqref{eq:OutputZeroingCondition} in~\eqref{eq:OCP_full} with applying the necessary conditions of optimality as first considering~\eqref{eq:OutputZeroingCondition} yields \eqref{eq:T-OptimizationProblem}, cf. Figure~\ref{fig:sketch}. %
While at first glance this result seems to be of purely technical nature, it is of interest of its own in context of turnpike analysis of OCPs. %
Recall that the main structure exploited in usual turnpike analysis is that %
KKT conditions for an optimal steady state of a system w.r.t.\ some stage cost~$\ell$ coincide with the steady-state conditions of the optimality system, %
see, e.g.,~\cite{Trelat15a,kit:zanon18a,tudo:faulwasser20h,Gruene2019}. %


\subsection{Optimal control on trim manifold $\mcl{T}$ with orthogonal forcing ($f_\theta\equiv0$)}\label{sec:trim_tp_general}

Whereas in the last section we allowed for a general forcing term $f(u) = \left( f_s(u),f_\theta(u)\right)^\top$ we will now consider particular forcings that preserve the structural symmetry of the system. As discussed in Remark~\ref{rem:orthForcing}, this means that $f_\theta\equiv0$ such that the forcing only acts orthogonal to the space spanned by symmetry variables, or, in other words, forcing is only allowed for the shape space variables.
%

We are interested in how optimal trims look like, i.e., triples ($s,v_\theta,u$) satisfying $T(s,v_\theta,u) =0$ (recall \eqref{eq:OutputZeroingCondition}) that minimize the running cost. Here we assume now that $\dot{s}\equiv v_s\equiv0$ as this is required for a trim according to Definition~\ref{def:Trims}.

Due to $f_\theta\equiv0$ and $v_s=0$ and as $\theta$ does occur neither in the cost functional nor on the right-hand side of the dynamics, we consider the steady state optimization problem\\

\noindent \fbox{
	\begin{minipage}{.98 \textwidth}
		\begin{align}\label{eq:EL_OCP_trim}\tag{SOP}
			\begin{split}
				\min_{(\bar{s},\bar{v}_\theta,\bar{u})\in \mathbb{R}^3} &\quad \ell(\bar{s},0,\bar{v}_\theta,\bar{u})\\
				\text{ subject to }
				T(\bar{s},\bar{v}_\theta,\bar{u}) & = 0 \text{ (as in \eqref{eq:OutputZeroingCondition})}.
			\end{split}
		\end{align}
\end{minipage}}

$\quad$\\
This problem is a particular version of \eqref{eq:T-OptimizationProblem} with $f_\theta\equiv0$.

To derive NCO, we define the Lagrangian with an adjoint state $\bar{\lambda}\in \mathbb{R}$
\begin{align}
	\mathcal{L}(\bar{s},\bar{v}_\theta,\bar{u},\bar{\lambda}) = \ell(\bar{s},0,\bar{v}_\theta,\bar{u}) + \bar{\lambda} T(\bar{s},\bar{v}_\theta,\bar{u})\label{eq:augLss}
\end{align}
and compute the stationarity conditions, omitting the argument $\bar{s}$ of the mass matrices and the potential:

\begin{subequations}
	\label{eq:sspL}
	\begin{align}
		0&=\frac{\partial \mathcal{L}}{\partial \bar{s}} = \,\frac{\partial \ell}{\partial \bar{s}} \hspace*{0.2cm}+ \bar{\lambda} \frac{\partial}{\partial \bar{s}} T(\bar{s},\bar{v}_\theta,\bar{u}) \label{eq:sspL1}\\
		0&=\frac{\partial \mathcal{L}}{\partial \bar{v}_\theta} = \frac{\partial \ell}{\partial \bar{v}_\theta} + \bar{\lambda} \frac{\partial}{\partial \bar{v}_\theta} T(\bar{s},\bar{v}_\theta,\bar{u}) \label{eq:sspL2}\\
		0&=\frac{\partial\mathcal{L}}{\partial {\bar{u}}} = \frac{\partial \ell}{\partial \bar{u}} \hspace*{0.2cm}+ \bar{\lambda} \frac{\partial}{\partial \bar{u}} T(\bar{s},\bar{v}_\theta,\bar{u})  \label{eq:sspL3}\\
		0&= \frac{\partial\mathcal{L}}{\partial {\bar{\lambda}}} = T(\bar{s},\bar{v}_\theta,\bar{u}) \label{eq:sspL4}.
	\end{align}

	%

\end{subequations}

\begin{proposition}[Correspondence of NCOs (cont'd)] \label{prop:NCOmatch2}
	Consider \eqref{eq:OCP_full} and its reduced counterpart \eqref{eq:EL_OCP_trim} for a Lagrangian system of type \eqref{eq:v_s} with $f_\theta \equiv 0$.
	Suppose that Assumptions~\ref{Ass:BlockDiagonalForcing} and~\ref{Ass:ell} hold.
	If an optimal solution and the corresponding Lagrange multiplier for \eqref{eq:OCP_full} satisfy $\dot{s}^*(t) =v^*_s(t) = 0$, $\dot{\lambda}^*_s(t) = 0$ and $\dot{\lambda}^*_{v_\theta}(t)=0$ for $t\in [t_1,t_2]\subseteq[0,T]$, then they also satisfy the dynamics of the first order NCOs of \eqref{eq:EL_OCP_trim} on $[t_1,t_2]$ when identifying $\bar{s}$ with $s^*$, $\bar{v}_\theta$ with $v_\theta^*$, $\bar{u}$ with $u^*$ and $\bar{\lambda}$ with $\lambda_{v_s}^*$. In particular, $v^*_\theta(t) = const.$, $\theta^*(t)$ is linear and $\lambda^*_{v_s}(t)=const$ for all $t\in [t_1,t_2]$.
\end{proposition}

\begin{proof}
	The proof follows the same structure as the proof of Proposition~\ref{prop:NCOmatch} setting $f_\theta =0$.
\end{proof}
The results of Propositions~\ref{prop:NCOmatch} and~\ref{prop:NCOmatch2} are summarized in Figure~\ref{fig:sketch}. In essence, these results extend the usual turnpike analysis~\cite{Carlson91,Trelat15a,kit:zanon18a}, in the sense that the turnpike now corresponds to the solutions living on the trim manifold $\mathcal{T}$. Moreover, note that even on the level of \eqref{eq:EL_OCP_trim} one considers situation more general than classical steady-state turnpikes. The reason is that, while \eqref{eq:EL_OCP_trim} is a steady-state problem, its optimal solution characterizes a trim, which in turn corresponds to a continuum of dynamic trajectories. 
\begin{figure}
	\includegraphics[width=0.99\textwidth]{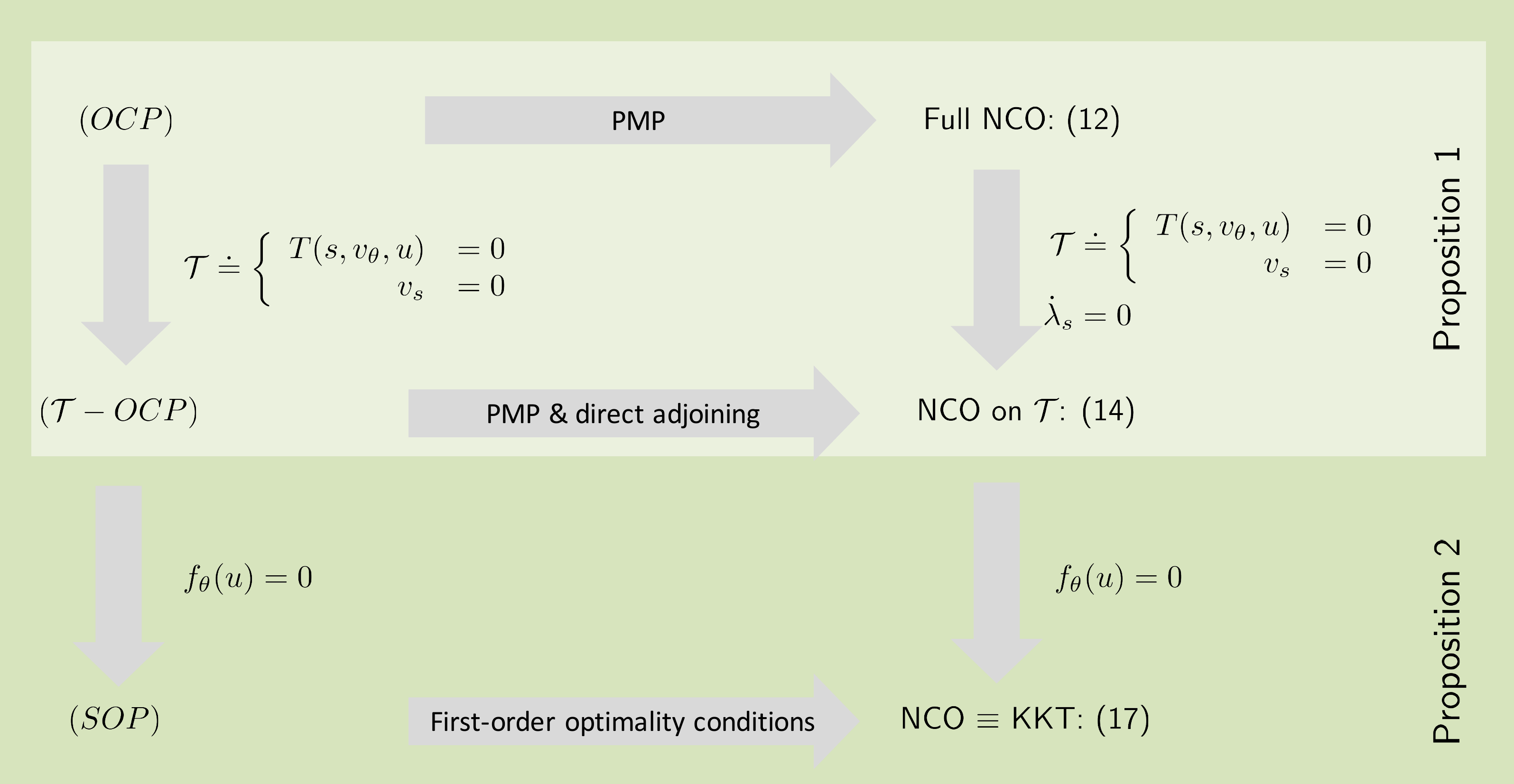}
	\caption{Illustration of Propositions \ref{prop:NCOmatch} and~\ref{prop:NCOmatch2}. \label{fig:sketch}}
\end{figure}

\begin{remark}[Trims, velocity turnpikes and the trim manifold]
	In~\cite{FaulFlas19,tudo:faulwasser20e} we proposed the concept of velocity turnpikes to establish a link between symmetries, trims and turnpike properties in OCPs. In essence, velocity turnpikes are a special case of the trim manifold approach considered in the present paper. 
	To establish velocity turnpikes we considered systems
	\begin{align*}
		\dot{q} & = v\\
		\dot{v} & = f(v,u),
	\end{align*}
	i.e.\ the dynamics were assumed to be invariant of all configuration velocities.
	Geometrically speaking, this system class corresponds to mechanical systems on Lie groups, i.e.\ all variables are cyclic.
	Since the stage cost $\ell$ was not allowed to depend on the cyclic variables either, one 
	can rely on established turnpike concepts which include all state variables except the position states.
	Based on the results of this paper, the next generalization is to remove the assumption on cyclic variables and to consider (mechanical) systems with general symmetries defined by left-actions $\Psi: \mathcal{G} \times TQ \to TQ$.
\end{remark}

While the analysis so far leveraged---at least partially---symmetries and trims, we turn to a more general setting of manifold turnpikes.

\section{Manifold turnpikes and optimal operation on a manifold}

The basis of our subsequent developments is the following extension of~\eqref{eq:OCP_full}

\noindent \fbox{
	\begin{minipage}{.98 \textwidth}
		\begin{align}
			\label{eq:OCP_full_con}\tag{OCP$_\psi$}
			\begin{split}
				\min_{u \in L^\infty([0,T],\mathbb{R}^{m})}\quad &\int_0^T \ell (s(t),v_s(t),v_\theta (t),u(t))\,\mathrm{d}t\\
				\text{subject to the system dynamics}&~\eqref{eq:v_s} 
				\nonumber \\
				(s\ v_s\ \theta\ v_\theta)(0) &= (s^0\ v_s^0\ \theta^0\ v_\theta^0),\\
				\psi\Big((s\ v_s\ \theta\ v_\theta)(T)\Big) &\leq 0,	
			\end{split}
		\end{align}
\end{minipage}}
$\quad$\\
where $\psi:\mathbb{R}^4 \to \mathbb{R}^{n_\psi}$ specifies a target set. For brevity, we use the shorthand 
\[x \doteq  \begin{bmatrix}s& v_s& \theta& v_\theta \end{bmatrix}^\top\]
for the state vector. 
Then, defining the terminal set $\Psi \doteq    \{x\in \mathbb{R}^4 \,|\, \psi(x) \leq 0\}$, allows to state the terminal constraint in~\eqref{eq:OCP_full_con} as $x(T) \in \Psi$.


\subsection{Sufficient conditions for manifold turnpikes}\label{Subsec:SC}
Next we introduce a general framework for the concepts of dissipativity, the turnpike property and optimal operation with respect to a manifold that is well-suited to our particular application with a trim manifold. 
Before we proceed, we require the following notation: A continuous function $\alpha: \mathbb{R}_{\geq 0} \rightarrow \mathbb{R}_{\geq 0}$ is said to be a class~$\mathcal{K}$-function if $\alpha(0) = 0$ and $\alpha$ is monotonically increasing, see also~\cite{Sont98,Kell14} for further details and explanations w.r.t.\ such comparison functions.

Next, we want to show that \eqref{eq:OCP_full_con} exhibits a measure turnpike property with respect to a manifold. To this end, we require both strict dissipativity as well as cost controllability whereas the latter can be replaced by (potentially) weaker assumptions referring to reachability. 
Note that we state the following definitions and Theorem~\ref{thm:ManifoldTurnpike} such that they are directly applicable to OCPs with more general system dynamics evolving in~$\mathbb{R}^{n}$, $n \in \mathbb{N}$, inputs $u \in \mathbb{R}^{m}$, and stage costs $\ell:\mathbb{R}^{n}\times \mathbb{R}^{m} \to\mathbb{R}$.


\begin{definition}[Strict dissipativity w.r.t.\ a manifold]\label{def:StrictDissipativity}
	\eqref{eq:OCP_full_con} %
	is said to be strictly dissipative w.r.t.\ a manifold $\mathcal{T}\subset\mathbb{R}^{n}$ on the set~$\mathbb{X} \subseteq \mathbb{R}^{n}$ %
	if there exists a storage function $S:\mathbb{R}^{n}\to \mathbb{R}_{\geq 0}$, that is bounded on compact sets, and a $\mathcal{K}$-function~$\alpha$ such that %
	\begin{align}\label{eq:sDI}
		\begin{split}
			S(x^\star (T))- S(x^0) \leq \int_0^T    \ell (x^\star(t),u^\star(t)) & -\alpha(\dist(x^\star (t),\mathcal{T}))\,\mathrm{d}t \quad\forall\,T \geq 0
		\end{split}
	\end{align}
	for all optimal controls $u^\star \in L^{\infty}([0,T],\mathbb{R}^{m})$ and %
	associated state trajectories $x^\star \in W^{1,\infty}([0,T],\mathbb{R}^{n})$ with initial values $x^0 \in \mathbb{X}$.
\end{definition}
We emphasize that in this defintion strict dissipativity is a property of \eqref{eq:OCP_full_con} which is parametric in $x^0$ and~$T$. Alternatively, it can be defined as a property of the underlying dynamical system, see the foundational works of Willems~\cite{Willems71a,Willems72a,Willems72b} or more recent treatments in~\cite{Willems07a,Moylan14a}. In addition to dissipativity, we further require certain reachability properties. Here, we state the results based on cost controllability on a compact set, see~\cite{Coro20} for details. Definition~\ref{def:CostControllability} extends the previously proposed concepts of cost controllability introduced in \cite{TunaMessinaTeel2006,Grun09,GrunPann10} for discrete-time systems and~\cite{ReblAllg12} in the continuous-time setting, see~\cite{WortRebl14,WortRebl15} for connections between continuous- and discrete-time systems and~\cite{Wort11} for a thorough comparison of cost controllability and its precursors.

In the following we will denote by $x(t;x^0,u)$ a state trajectory evolving from the dynamics \eqref{eq:v_s} with control $u$ and initial state $x^0$, where with slight abuse of notation we stick to the state space being $\mathbb{R}^{n}$ instead of $\mathbb{R}^4$.

\begin{definition}[Cost controllability on a set~$\mathbb{X}$]\label{def:CostControllability}
	Let a set~$\mathbb{X} \subseteq \mathbb{R}^{n}$ be given. Then, defining 
	\[\ell^\star(x^0) \doteq \inf_{u \in \mathbb{R}^{m}} \ell(x^0,\hat{u}),
	\]
	\eqref{eq:OCP_full_con}
	with (optimal) value function~$V_T: \mathbb{R}^{n} \rightarrow \mathbb{R} \cup \{-\infty,\infty\}$ defined by
	\begin{equation}\nonumber
		V_T(x^0) \doteq    \inf_{u \in L^\infty([0,T],\mathbb{R}^{m})} \int_0^T \ell(x(t;x^0,u),u(t))\,\text{d}t
	\end{equation}%
	is called cost controllable %
	if there exists a bounded and increasing growth function $B_{\mathbb{X}}:\mathbb{R}_{\geq 0} \to \mathbb{R}_{\geq 0}$ such that
	\begin{align}\label{eq:costcon}
		V_T(x^0) \leq B_{\mathbb{X}}(T) \cdot \ell^\star(x^0) \qquad\forall\,x^0 \in \mathbb{X}, T \geq 0.
	\end{align}
\end{definition}
The strict dissipativity inequality \eqref{eq:sDI} implies that the cost of optimal trajectories is bounded from below by the distance to the trim manifold. In particular, this means that cost controllability in the sense of \eqref{eq:costcon} implies that the manifold $\mathcal{T}$ can be approached from any initial value $x_0\in \mathbb{X}$. 
For the sake of clarity of presentation, we assumed boundedness of the growth function~$B_{\mathbb{X}}$. However, we note that often in the context of cost controllability, the bound $\limsup_{T\rightarrow \infty} B_{\mathbb{X}}(T)/T < 1$ suffices, see~\cite{MullWort17} for details. 
\begin{definition}[Manifold turnpike property]
	\eqref{eq:OCP_full_con}
	is said to have the manifold turnpike property w.r.t.\ a manifold~$\mathcal{T}$ on a set $\mathbb{X} \subseteq \mathbb{R}^{n}$ %
	if, for all compact sets $K \subseteq \mathbb{X}$ and for all $\varepsilon > 0$, there is a constant $C_{K,\varepsilon}$ such that %
	all optimal solutions consisting of~$u^\star \in L^\infty([0,T],\mathbb{R}^{m})$ and $x^\star(\cdot;x^0,u^\star) \in W^{1,\infty}([0,T],\mathbb{R}^n)$ satisfy
	\begin{align*}
		\mu\Big(\{t \in [0,T]\,| \dist(x^\star(t;x^0,u^\star),\mathcal{T})>\varepsilon\}\Big) \leq C_{K,\varepsilon} \qquad\forall\,x^0 \in K, T>0
	\end{align*}
	where $\mu(S)$ denotes the Lebesgue-measure of a set~$S \subseteq \mathbb{R}^{n}$.
\end{definition}
Next, we can state our main theorem of this section; namely that strict dissipativity and cost controllability imply the turnpike property on manifolds.

\begin{theorem}\label{thm:ManifoldTurnpike}
	Let \eqref{eq:OCP_full_con} be strictly dissipative
	w.r.t.\ a manifold $\mathcal{T} \subseteq \mathbb{R}^{n}$ on~$\mathbb{X} \subseteq \mathbb{R}^4$ and cost controllable on~$\mathbb{X}$. %
	Then, the OCP satisfies the manifold turnpike property w.r.t.~$\mathcal{T}$ on $\mathbb{X}$.
\end{theorem}
\begin{proof}
	Let $K \subseteq \mathbb{X}$ be an arbitrary but fixed compact set. %
	Moreover, for a given initial value~$x^0 \in K$, let $(x^\star,u^\star)$ be an optimal solution of the OCP in consideration. %
	In view of the assumed strict dissipativity~\eqref{eq:sDI}, we obtain
	\begin{align}\label{eq:DI1}
		\begin{split}
			\int_0^T \alpha(\dist(x^\star(t),\mathcal{T}))\,\text{d}t &\leq S(x^0) - S(x(T)) + \int_0^T \ell(x^\star(t),u^\star(t)) \,\text{d}t \\
			&\leq c_1(K) + \int_0^T \ell(x^\star(t),u^\star(t)) \,\text{d}t.
		\end{split}
	\end{align}
	where $c_1(K)\doteq   \sup_{x^0\in K}S(x^0)$ by boundedness on the compact set~$K$ and positivity of the storage function~$S$. %
	Next, we estimate the last term using cost controllability~\eqref{eq:costcon}, i.e., 
	\begin{align*}
		\int_0^T \ell(x^\star(t),u^\star(t)) \text{d}t \leq B_{\mathbb{X}}(T) \cdot \ell^\star(x^0) \leq c_2(K)
	\end{align*}
	for a constant $c_2(K)\geq 0$ independent of $T$, %
	where the last inequality follows by the Weierstra\ss{}-Theorem based on continuity of~$\ell$ and compactness of the set~$K$ 
	and boundedness of the growth function~$B_{\mathbb{X}}$. %
	Further, setting $S_\varepsilon = \{t\in [0,T]\,|\, \dist(x^\star(t),\mathcal{T})>\varepsilon\}$, we get the estimate
	\begin{align*}
		& \int_0^T \alpha(\dist(x^\star(t),\mathcal{T}))\,\text{d}t \\
		= & \int_{S_\varepsilon} \alpha(\dist(x^\star(t),\mathcal{T}))\,\text{d}t 
		+ \int_{[0,T]\setminus S_\varepsilon} \alpha(\dist(x^\star(t),\mathcal{T}))\,\mathrm{d}t \geq \mu(S_\varepsilon) \cdot \alpha(\varepsilon)
	\end{align*}
	using monotonicity of the $\mathcal{K}$-function~$\alpha$. Combining the derived inequalities, 
	this yields the upper bound $C_{K,\varepsilon} \doteq    (c_1(K)+c_2(K))/\alpha(\varepsilon)$ on~$\mu(S_\varepsilon)$, which concludes the proof. 
\end{proof}

\subsection{Optimal operation on a manifold}\label{Subsec:TrimManifold}

Besides establishing turnpike theorems, dissipativity can be used to deduce performance results for model predictive control %
and to characterize the asymptotics of optimal control problems. %
In this context, optimal operation at the turnpike, in our case a manifold, plays an important role.

To the end of analyzing such links, we propose the following definition of optimal operation on a manifold, which is a natural extension of the established concept of optimal operation at steady state, cf.~\cite{Angeli2012}. We assume in the following that the cost functional is zero on the manifold. This can always be achieved, if the stage costs are constant on the manifold, e.g.\ by subtracting the constant.
\begin{definition}[Optimal operation on the trim manifold $\mathcal{T}$]\label{def:OptimallyOperated}
	Let the stage cost be zero on the trim manifold~$\mathcal{T}$, which was defined in Definition~\ref{rem:TrimManifold}, 
	i.e., $\ell(s,v_s,v_\theta,u) = 0$ if $T(s,v_\theta,u) = 0$ holds (in particular, $x = (s\,v_s\,\theta\,v_\theta)^\top \in \mathcal{T}$). %
	Then, System~\eqref{eq:v_s} is said to be optimally operated on the manifold $\mcl{T} \subseteq \mathbb{R}^4$ %
	if, for all $x^0 \in \mathbb{R}^4$ and all 
	control-state pairs $(x,u) \in W^{1,\infty}([0,\infty),\mathbb{R}^4) \times L^\infty([0,\infty),\mathbb{R}^{m})$, %
	the following inequality holds
	\begin{equation}\label{eq:OptOpSS}
		\liminf_{T \to \infty}\ \frac 1 T \int_0^T \ell (s(t),v_s(t),v_\theta (t),u(t))\,\mathrm{d}t \geq 0.
	\end{equation}
\end{definition}

Optimally operated at the manifold means that the averaged costs for any trajectory is at least as high as the averaged cost on the manifold in the limit.
\begin{proposition}[Optimal operation on  $\mathcal{T}$]\ \\%
	Let \eqref{eq:OCP_full_con} be strictly dissipative on a set~$\mathbb{X} \subseteq \mathbb{R}^4$ %
	such that the storage function~$S$ is bounded on the terminal region~$\Psi$. %
	Moreover, let the stage cost~$\ell$ be zero on the manifold~$\mcl{T} \subseteq \mathbb{R}^4$ as assumed in Definition~\ref{def:OptimallyOperated}. %
	Then, the system governed by the dynamics~\eqref{eq:v_s} is optimally operated %
	on~$\mathcal{T}$.
\end{proposition}
\begin{proof}
	The proof follows along the lines of~\cite[Theorem 3]{Faulwasser2017} by contradiction. %
	Let $x^0 \in \mathbb{R}^4$ be given. %
	Suppose that there exists a monotonically increasing sequence $\{T_k\}_{k\in \mathbb{N}}$, $T_k\to \infty$ for $k\to \infty$, %
	with $(x^k,u^k) \in W^{1,\infty}([0,T_k],\mathbb{R}^4) \times L^\infty([0,T_k],\mathbb{R}^{m})$ admissible for the OCP in consideration such that 
	\begin{align}\label{eq:ProofOptimallyOperated}
		\liminf_{k \to \infty}\ \frac 1 {T_k} \int_0^{T_k} \ell(s^k(t),v_s^k(t),v_\theta^k(t),u^k(t))\,\text{d}t \leq -\delta
	\end{align}
	for some $\delta >0$.
	Dividing the strict dissipativity inequality~\eqref{eq:sDI} by $T_k$ yields
	\begin{align*}
		\frac{1}{T_k} \left(S(x^\star_k (T_k))- S(x^0)\right) \leq \frac{1}{T_k} \int_0^{T_k} \ell(x^\star_k(t),u^\star_k(t)) & -\alpha(\dist(x^\star_k(t),\mathcal{T}))\,\mathrm{d}t
	\end{align*}
	where $x^\star_k = x^\star_k(\cdot;x^0,u^\star_k)$ solves %
	the OCP~\eqref{eq:OCP_full_con} 
	with optimization horizon~$T_k$ %
	(feasibility is ensured by the existence of the sequence~$(T_k)_{k \in \mathbb{N}}$) and, with a slight abuse of notation, %
	$\ell(x^\star_k(t),u^\star_k(t))$ denotes the stage cost evaluated along the corresponding optimal control-state pair. %
	Invoking the assumed boundedness of~$S$ on the terminal region, the difference $S(x^\star_k(T))- S(x^0)$ is bounded and %
	the left hand side equals zero. Hence, taking non-negativity of~$\alpha$ and optimality of the control-state pair $(u^\star_k,x^\star_k)$ into account, %
	yields non-negativity of the left hand side of inequality~\eqref{eq:ProofOptimallyOperated} and, thus, $0 \leq -\delta$, i.e.\ the desired contradiction.  
\end{proof}

\begin{remark}[Link to overtaking optimality]
	A more classical concept, which originated in the analysis of infinite-horizon OCPs, is overtaking optimality~\cite{Carlson91,Carlson90a}. 	That is, instead of~\eqref{eq:OptOpSS}
	consider 
	\begin{equation}
		\liminf_{T \to \infty}  \int_0^T \ell (s(t),v_s(t),v_\theta (t),u(t))\,\mathrm{d}t  -  \int_0^T \ell (\bar s,0,\bar v_\theta (t),\bar u(t))\,\mathrm{d}t  \geq 0.
	\end{equation}
	If this condition is considered for a fixed initial condition, it gives the concept of overtaking optimality. If it is considered for a set of initial conditions, it defines a generalized concept of optimal operation being characterized by $(\bar s,0,\bar v_\theta (t),\bar u(t))$.
	We refer to~\cite[Chap. 4]{Pirkelmann2020} for further (discrete-time) insights on the link between overtaking optimality and optimal operation. 
\end{remark}


\section{Hamiltonian perspective and Legendre transformation of OCPs}
\label{sec:trim_tp:H}

The Euler-Lagrange equations~\eqref{ELStandard}   
can alternatively be written in Hamiltonian form, i.e.,\ $\dot{q} = \partial H/ \partial p, \; \dot{p} = -\partial H/ \partial q + f(u)$, by means of the Hamiltonian
\begin{equation} \label{eq:HamiltonFunction} 
	H(q,p) = \frac 12 {p}^\top M^{-1}(q) {p} + V(q),
\end{equation}
where $M^{-1}$ is the inverse of the mass matrix.
With configuration variables $q\in Q$, the corresponding momenta $p$ lie in the cotangent space $T_q^*Q$.
As $M(q)$ is assumed to be regular, the Lagrange function \eqref{eq:LagrangeFunction} and the Hamiltonian \eqref{eq:HamiltonFunction} are hyperregular and Euler-Lagrange and Hamilton equations are equivalent.
The Legendre transform of \eqref{eq:LagrangeFunction} gives the relation $p=M(q)\dot{q}$. 
In particular, under Assumption~\ref{Ass:BlockDiagonalForcing}, the Hamilton equations in shape and cyclic variables are
\begin{align}\label{eq:HBlockDiag_ps}\tag{H}
	\begin{split}
		\dot{s}& = M_{11}^{-1}(s) p_s \\
		\dot{p}_s & =  - \frac 12 M_{11}^{-1}(s)^\prime p_s^2 - \frac 12 M^{-1}_{22}(s)^\prime p_\theta^2 - V^\prime(s) + f_s(u)   \\
		\dot{\theta} & = M_{22}^{-1}(s) p_\theta  \\
		\dot{p}_\theta & = f_\theta(u)   
	\end{split}
\end{align}
with $p_s = M_{11}(s) v_s, \; p_\theta = M_{22}(s) v_\theta$.

Assuming orthogonal forcing, i.e.\ $f_\theta \equiv 0$, the last equation of~\eqref{eq:HBlockDiag_ps} directly gives the conserved quantity induced by the symmetry, namely $p_\theta = M_{22}(s) v_\theta = const.$
Recall the characterization of trim primitives in Lemma~\ref{LemmaTrimPrimitives} via the forced amended potential \eqref{Eq:ForcedAmendedPotential}, which yields for \eqref{eq:HBlockDiag_ps}
\begin{align}\label{eq:trimCritPointVumu_2}
	\nabla_s V_\mu^u(s,\mu,u) = V^\prime(s) + \frac 12 {M^{-1}_{22}}(s)^\prime \mu^2 -f_s(u) = 0
\end{align}
with $\mu = p_\theta = M_{22}(s) v_\theta$.
Corollary~\ref{cor:VmuT} states that $\nabla_s V_\mu^u(s,\mu,u) = T(s, M_{22}^{-1}p_\theta,u)$ and thus, $\nabla_s V_\mu^u(s,\mu,u)$ can alternatively be used to define the trim manifold.

With $\tilde{\ell}(s,p_s,\theta,p_\theta)\doteq   \ell(s,M_{11}(s)v_s,\theta,M_{22}(s)v_\theta)$
, the optimal control problem
\eqref{eq:OCP_full}, the reduced problem on the trim manifold \eqref{eq:T-OptimizationProblem}, as well as the steady state optimization problem \eqref{eq:EL_OCP_trim} can alternatively be stated in the Hamiltonian setting, i.e., \ replacing the corresponding Euler-Lagrange equations by the Hamiltonian counterparts.
In complete analogy, first-order necessary conditions for optimality can be derived and compared in order to see that the problems lead to identical solutions if the full optimal control problem is constrained to solutions satisfying $\nabla_s V_\mu^s =0$. Moreover, as we will show next, based on the knowledge of the underlying coordinate change one can also derive a coordinate change for the adjoints.\footnote{Note that here we consider a Legendre transformation relating the mechanical Hamiltonian~\eqref{eq:HamiltonFunction} to the mechanical Lagrangian~\eqref{eq:LagrangeFunction}. One could as well consider a  Legendre transformation of the optimal control Hamiltonian, this leads to the Lax-Hopf formulas and related approaches, see~\cite{Claudel10a,Claudel10b}.} 

\subsection{Legendre-induced transformation of adjoints}
The diffeomorphism $\Phi:T_qQ \to T_q^*Q$ given by
\[
\Phi: \begin{bmatrix}
	s \\ v_s \\ \theta \\ v_\theta
\end{bmatrix} \mapsto \begin{bmatrix}
	s \\ p_s \\ \theta \\ p_\theta
\end{bmatrix} = \begin{bmatrix}
	s \\ v_s M_{11}(s) \\ \theta \\ v_\theta M_{22}(s)
\end{bmatrix} 
\]
maps \eqref{eq:ELBlockDiag_vs} to \eqref{eq:HBlockDiag_ps}. Put differently, the coordinate change $\Phi$ is induced by the by the Legendre transformation of \eqref{eq:LagrangeFunction} to  \eqref{eq:HamiltonFunction}.
\begin{lemma}[Legendre-induced adjoint transformation]
	Consider \eqref{eq:OCP_full} and let\\ $(x^\star, u^\star, \lambda_0^\star, \lambda^\star)$ be an optimal lift. Consider a diffeomorphic coordinate change $\Phi: \mathbb{R}^{n} \to \mathbb{R}^{n}, x\mapsto z$ valid along any optimal solution of \eqref{eq:OCP_full}. Suppose that \eqref{eq:OCP_full} is normal and set $\lambda_0^\star =1$. Let $\nu$ be the adjoint corresponding to \eqref{eq:OCP_full} expressed in the coordinates $z = \Phi(x)$. Then the corresponding adjoints satisfy
	\begin{equation} \label{eq:AdjointTrafo}
		\left.\left(\frac{\partial \Phi}{\partial x}\right)^\top\right|_{x = \Phi^{-1}(z)} \nu^\star = \lambda^\star \quad 
	\end{equation}
	with 
	\begin{equation}  \label{eq:AdjointTrafoII}
		\left.\left(\frac{\partial \Phi}{\partial x}\right)^\top\right|_{x = \Phi^{-1}(z)}
		= \begin{bmatrix}
			1 & M_{11}'v_s & 0 & M_{22}'v_\theta \\
			0 & M_{11} &0 &0 \\
			0 & 0 & 1 & 0 \\
			0 & 0 & 0 & M_{22}
		\end{bmatrix}
		=
		\begin{bmatrix}
			1 & M_{11}'M_{11}^{-1} p_s & 0 & M_{22}'M_{22}^{-1} p_\theta \\
			0 & M_{11} &0 &0 \\
			0 & 0 & 1 & 0 \\
			0 & 0 & 0 & M_{22}
		\end{bmatrix}
	\end{equation}
\end{lemma}
\begin{proof}
	Recall the optimal control Hamiltonian of \eqref{eq:OCP_full} written in $x-u$ coordinates
	\[
	{\mathcal{H}}(x,u,\lambda_0, \lambda) = \lambda_0\ell(x,u) +\lambda^\top f(x,u).
	\]
	Notice that in $z-u$ coordinates the dynamics $\dot x = f(x,u)$ read
	\[
	\dot z  = \frac{\partial \Phi}{\partial x}  f(x,u) = \left.\frac{\partial \Phi}{\partial x}\right|_{x = \Phi^{-1}(z)} f\left( \Phi^{-1}(z), u\right) \doteq g(z,u).
	\]
	Now, consider the optimal control Hamiltonian of \eqref{eq:OCP_full} expressed in $z-u$ coordinates
	\[
	\tilde{\mathcal{H}}(z,u,\nu_0, \nu) =  \nu_0\tilde \ell(z,u) + \nu^\top g(z,u).
	\]
	Substituting the expression for $g(z,u)$ and $\tilde\ell(z,u) \doteq \ell(\Phi^{-1}(z), u)$ yields
	\[
	\tilde{\mathcal{H}}(z,u,\nu_0, \nu) =  \nu_0\ell(\Phi^{-1}(z),u) + \nu^\top \left.\frac{\partial \Phi}{\partial x}\right|_{x = \Phi^{-1}(z)} f\left( \Phi^{-1}(z), u\right).
	\]
	As we assume that \eqref{eq:OCP_full} is normal, we set $\lambda_0^\star =1 = \nu_0^\star$. Comparing the last equation with the one for ${\mathcal{H}}(x,u,\lambda_0, \lambda)$ gives the first part of the assertion. The expression for $\left(\frac{\partial \Phi}{\partial x}\right)^\top$ follows from the definition of the coordinate change $\Phi$. 
\end{proof}
We remark that the transformation of adjoints does mainly rely on $\Phi$ being a diffeomorphic coordinate change along optimal solutions. That is, it can be applied to general OCPs. It is the structure of the matrix \eqref{eq:AdjointTrafoII} which is induced by the underlying Legendre transformation.

\section{Kepler Problem: Optimal operation on manifold turnpike}\label{sec:Kepler}
In astrodynamics, n-body problems are widely used to describe the dynamics of bodies in the gravitational field.
The two-body problem is also known as the Kepler problem and might be used to describe the motion of a spacecraft relative to a planet's  or a moon's gravitational field (ignoring all other influences from more distant bodies).
Coordinates can be chosen to describe the motion of the second body, relatively to the first body's motion, via radius $s \in \mathbb{R}_{>0}$ and angle $\theta \in [0,2\pi)$. 
With $v_s$ and $v_\theta$ denoting the corresponding velocities, the Lagrangian is given by
\begin{align*}
	L(s,v_s,\theta,v_\theta) = \frac 12 m_2 \left(v_s^2 + s^2 v_\theta^2\right) + \gamma \frac {m_1 m_2}{s} %
\end{align*}	
with $m_1,m_2$ being the masses of the primary (e.g.\ planet) and the secondary (e.g.\ spacecraft) body and $\gamma$ the gravitational constant.
As in \cite{SaakeOberBl}, we choose $k \doteq \gamma m_1 m_2 = 1.016895192894334 \cdot 10^3$ and $m_2  = 1.0$.

We have $L$ being independent of $\theta$, so this is a cyclic variable.
For $f_\theta(u)=0$, the conserved quantity is $p_\theta = m_2 s^2 v\theta$ (cf. Remark~\ref{rem:orthForcing}).
Further, Assumption~\ref{Ass:BlockDiagonalForcing} is satisfied, since the mass matrix is
\[ M = \begin{bmatrix}  m_2 & 0 \\ 0 & m_2 s^2 \end{bmatrix}, \text{ in particular, } M_{11}(s) = m_2, \, M_{22}(s) = m_2 s^2,\]
and $V(s) = -\gamma \frac {m_1 m_2}{s}$.
Note that the Kepler problem is special in the fact that $M_{11}$ is constant and thus, $M_{11}' = 0$.
Moreover, the model has a singularity at $s=0$, so we restrict to $s>0$. Then, $M_{22}\neq 0$ holds.

The Euler-Lagrange equations, directly written in first-order form, are
\begin{align*}
	\dot{s} & = v_s, &
	\dot{v}_s & = s v_\theta^2 - \frac {\gamma m_1 m_2}{m_2 s^2}  + \frac{1}{m_2} f_s(u),\\
	\dot{\theta} & = v_\theta,  &
	\dot{v}_\theta & =  - \frac{2}{s} v_\theta v_s  + \frac{1}{m_2 s^2} f_\theta(u),
\end{align*}
with control $u = (u_s,u_\theta)^\top \in \mathbb{R}^2$ and forcing $f_s(u) = u_s$, $f_\theta(u) = u_\theta$.

For the function $T$ of Lemma~\ref{lemma:T}, we obtain
\[ T(s,v_\theta,u) = s v_\theta^2 - \frac{k}{m_2 s^2} + \frac{1}{m_2} f_s(u), \]
i.e.\ any triple $(\tilde{s},\tilde{v_\theta},\tilde{u})$ such that $T(\tilde{s},\tilde{v_\theta},\tilde{u})=0$ generates a trim primitive.
Geometrically, trim primitives are circular motions of body $m_2$ about $m_1$.
The trim manifold reads
\[\mcl{T} =   \left\{ (s, v_s, \theta, v_\theta)^\top \in TQ \,\left|\, v_s =0,\,  u_s = -m_2 s v_\theta^2 + \frac{k}{s^2} \right. \right\}.\]

We consider optimal control problems of type \eqref{eq:OCP_full} on different, but always fixed, time horizons $T>0$.
The starting point is defined as $(s^0,v_s^0,\theta^0,v_\theta^0) = (5.0,0.0,0.0,$\\ $\sqrt{\frac{k}{m_2 5.0^3}})$; this corresponds to a trim primitive with zero control $u_s$.
We set terminal constraints $$
(s,v_s,v_\theta)(T) \doteq (s^f,v_s^f,v_\theta^f) = (6.0,0.0,\sqrt{\frac{k}{m_2 6.0^3}})$$
but do not constrain $\theta(T)$, since there is no need to fix the exact point on the uncontrolled periodic orbit which is defined by $(s^f,v_s^f,v_\theta^f)$.

Furthermore, let $\tilde{x} \doteq (\tilde{s},\tilde{v}_s,\tilde{\theta},\tilde{v_\theta}) = (4.5,0.0,0.0,\sqrt{\frac{k}{m_2 4.5^3}})$ be given and note that $T(\tilde{s},\tilde{v}_\theta,\tilde{u})=0$ holds for $\tilde{u} = (0.0,0.0)$, i.e.\ this defines an uncontrolled trim primitive.

\begin{figure}
	\centering
	\includegraphics[width=\textwidth]{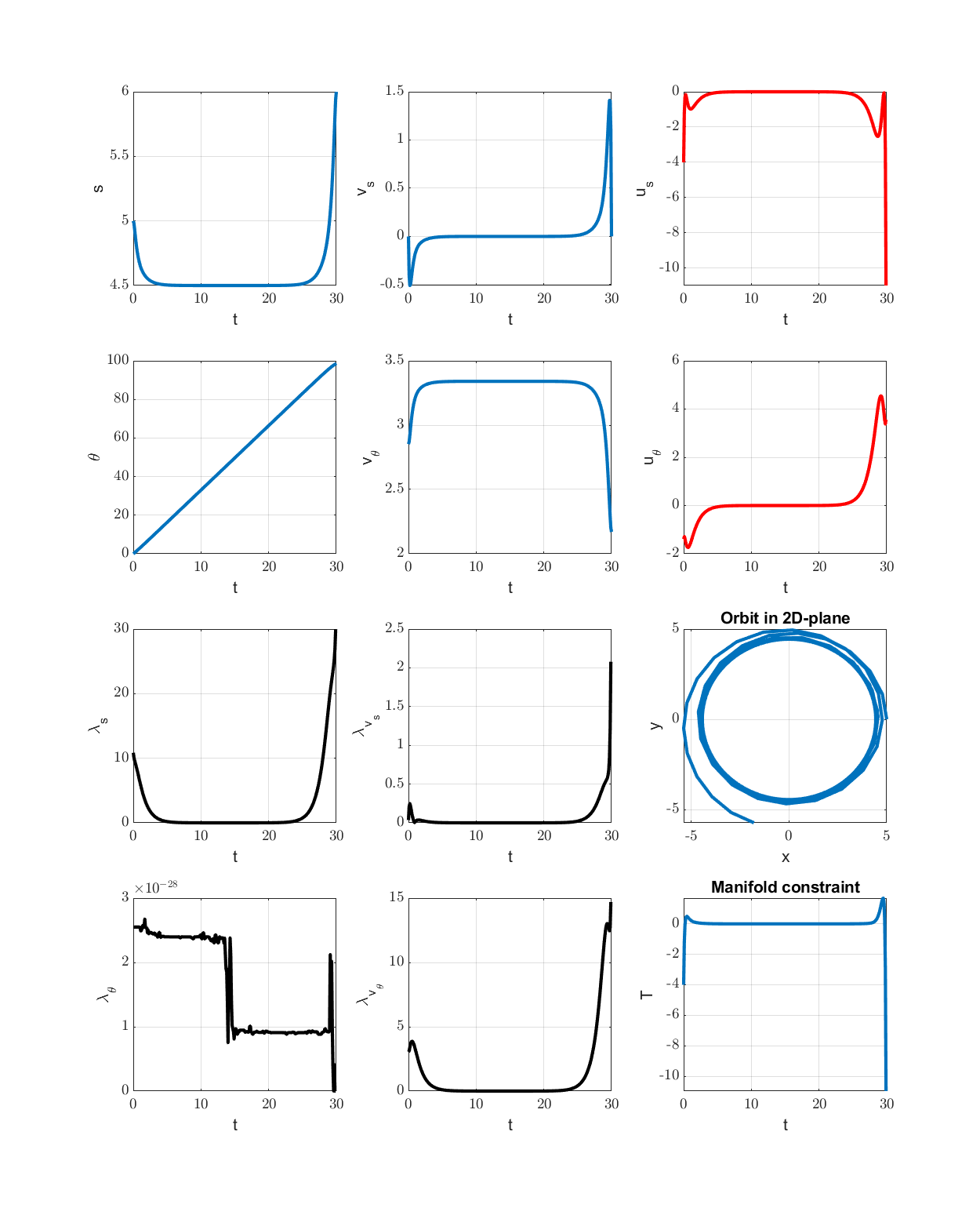}
	\caption{Example with turnpike on a trim at $s=4.5$ for quadratic cost functional as in \eqref{eq:CostsQR} on time horizon $T=30$.}
	\label{fig:QRTrim}
\end{figure}

Firstly, let us consider the cost functional 
\begin{equation} \label{eq:CostsQR}
	\ell(s,v_s,v_\theta,u) = \frac{1}{2} (x-\tilde{x})^\top Q (x-\tilde{x}) + (u-\tilde{u})^\top R (u-\tilde{u})
\end{equation}
with $Q=diag([1,0,1,1])$ and $R=10^{-2}\cdot diag([1,1])$.

We solve the corresponding (OCP) with CasADI, using a direct method with the RK-4 integrator for a discretization with $300$ nodes on a time interval with $T=30$. The result is given in Figure~\ref{fig:QRTrim}.
A turnpike can be observed at $(\tilde{s},\tilde{v}_\theta,\tilde{u}_s) = (4.5,\sqrt{\frac{k}{m_2 4.5^3}},0.0)$ with $v_s = v_\theta=u_\theta =0$ and all adjoints vanishing, too, for the largest part of the time interval.
The incoming and outgoing arcs are caused by the boundary conditions.
Mechanically, the solution corresponds to a circular-shaped turnpike orbit in the $2D$-plane, which is an element of the trim manifold.\\

While the first example specifically favors the $s=4.5$-orbit by construction, we now consider running costs which are designed using the general trim manifold description, i.e.\ 
\begin{equation} \label{eq:CostsTrim}
	\ell(s,v_s,v_\theta,u) = 5\cdot10^3 \cdot  T(s,v_\theta,u)^2 + \frac{1}{2} (s-\tilde{s})^2 + \frac{1}{2} (u-\tilde{u})^\top R (u-\tilde{u})
\end{equation}
with $\tilde{s} = 5.3$, $\tilde{u} = [0,1]^\top$, $R=10^{-3}\cdot diag([1,1])$.
Thus, the first term of $\ell$ vanishes whenever the system is on $\mathcal{T}$.
This criterion is complemented by the other two terms with arbitrarily chosen values of $\tilde{s}$ and $\tilde{u}$.
Setting the initial condition to $(s^0,v_s^0,\theta^0,v_\theta^0) = (5.3,0.0,0.0,\sqrt{\frac{k}{m_2 5.3^3}})$ makes an incoming arc obsolete; the system stays in the trim that is defined by the initial point almost until the end of the time interval, when the term $(u-\tilde{u})^\top R (u-\tilde{u})$ of $\ell$ in \eqref{eq:CostsTrim} rules the optimal solution.
This can be observed in Figure~\ref{fig:TOCP}, in which we show the computed solution for $T=100$ (RK4-integrator with $200$ discretization nodes).
Moreover, we depict the solution of \eqref{eq:T-OptimizationProblem}, which we have solved with CasADI, as well, using identical initial values for $(\bar{\theta},\bar{v}_{\bar{\theta}},\bar{u})$ and time horizon, with dashed lines.
Recall that in \eqref{eq:T-OptimizationProblem},  $(\bar{\theta},\bar{v}_{\bar{\theta}},\bar{u})$ are the dynamic states and controls, while $\bar{s}$ is a scalar parameter and $T(\bar{s},\bar{v}_{\bar{\theta}},\bar{u})=0$ is added as a nonlinear constraint.
For both problems, the same turnpike is approached, as can be seen in Figure~\ref{fig:TOCP} in the subfigures of the states and controls.
However, the adjoints for $s$ show different behavior, since in \eqref{eq:OCP_full}, there is an initial condition on $s$, while in \eqref{eq:T-OptimizationProblem}, there is not.
Further numerical discrepancies between the adjoints presumably stay in context with the accuracy of which $T=0$ is fulfilled when either considered within the objective (in \eqref{eq:OCP_full}) or as an equality constraint (in \eqref{eq:T-OptimizationProblem}).
Note that we do not consider terminal constraints in this example in order to match the setting of Proposition~\ref{prop:NCOmatch}.

\begin{figure}
	\centering
	\includegraphics[width=\textwidth]{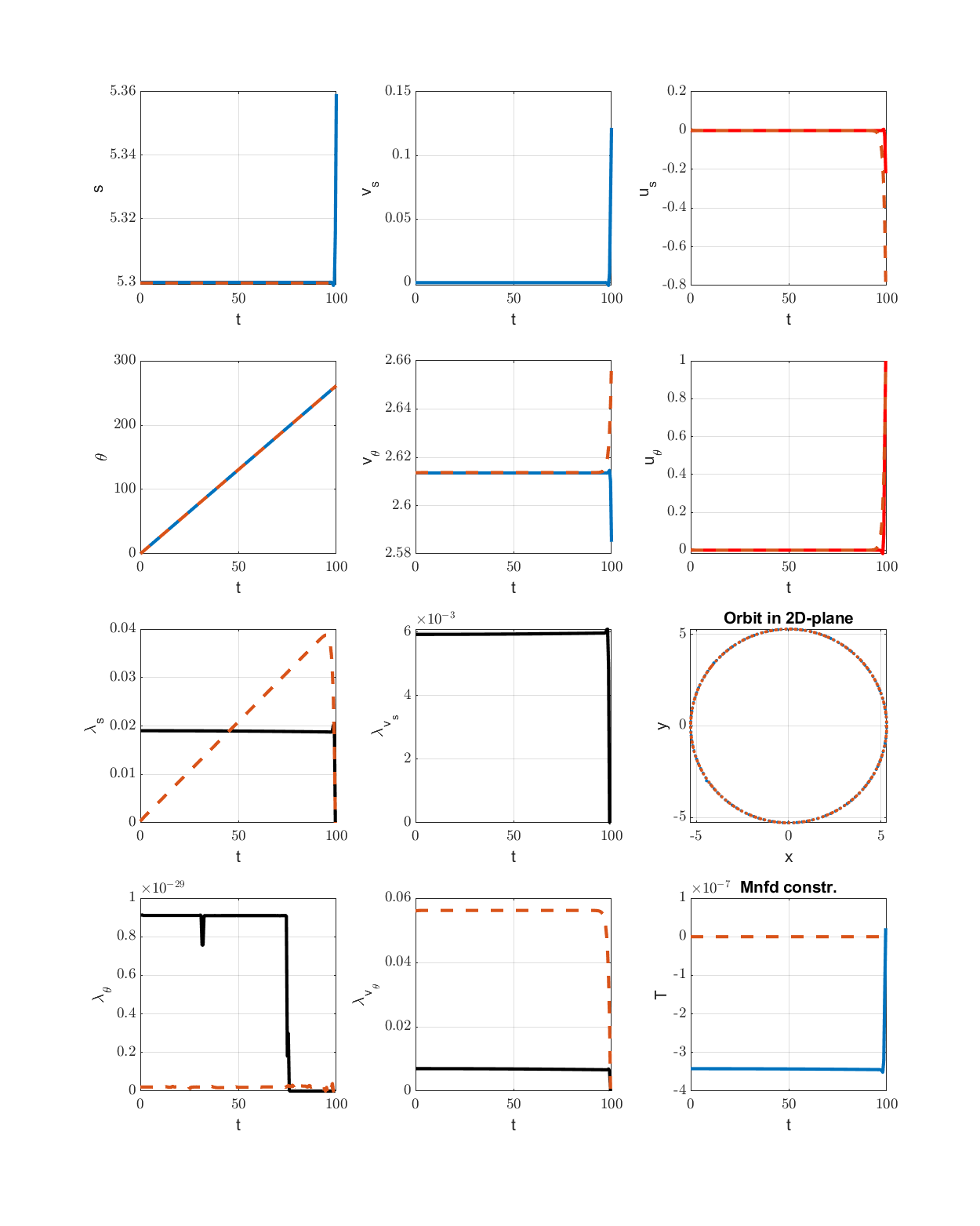}
	\caption{Example with turnpike on a trim obtained from running costs \eqref{eq:CostsTrim} on time horizon $T=100$.}
	\label{fig:TOCP}
\end{figure}

\section{Conclusions and Outlook}
The paper has studied the link of turnpikes, trim solutions, and symmetries in OCPs for mechanical systems. 
Specifically, we considered Lagrangian systems with symmetries. Based on the established concepts of trim solutions, we have shown that if either one first formulates the OCP and then applies the trim condition to the optimality system, or one first applies the trim condition and then formulates a reduced OCP, one obtains the same result. This generalizes a classical insight, wherein turnpikes are characterized as the attractive equilibria of the optimality system. Hence, the paper provides a novel characterization of time-varying---not necessarily periodic---turnpike solutions via reduced OCPs. Moreover, we introduced a notion of dissipation of optimal solutions with respect to the distance to a manifold (here the trim manifold) which implies that optimal system operation occurs on this manifold. The paper has also shown that the very same dissipativity condition implies the existence of a measure turnpike with respect to the trim manifold, i.e., the optimal solutions will spend only limited amount of time  far from this manifold. 
In sum, the present paper  introduced a novel manifold generalization of the established dissipativity notion for OCPs. This way it addresses the gap between the symmetry-based analysis of OCPs of mechanical systems and dissipativity-based turnpike analysis. 

Future work should discuss how the developed notions can be leveraged in context of receding-horizon optimal control. Moreover, it would be interesting to generalize the concept of manifold turnpikes even further and to consider symmetries induced by non-mechanical systems.

\bibliographystyle{ieeetr}      

\bibliography{FaulFlas20}   

\end{document}